\documentclass{amsart}
\usepackage{mathtools}
\usepackage{amsmath,amssymb,amsthm,mathrsfs,array,color}
\usepackage{verbatim}
\usepackage{url}

\newcommand{\re}{\mathbb{R}}
\newcommand{\mR}{\mathbb{R}}

\newcommand{\N}{\mathbb{N}}

\newcommand{\half}{\frac{1}{2}}
\newcommand{\lmd}{\lambda}

\newcommand{\nn}{\nonumber}
\newcommand{\eps}{\epsilon}

\newcommand{\dt}{\delta}

\def\af{\alpha}

\newcommand{\sig}{\sigma}

\newcommand{\reff}[1]{(\ref{#1})}

\newcommand{\mc}[1]{\mathcal{#1}}
\newcommand{\mt}[1]{\mathtt{#1}}

\newcommand{\supp}[1]{\mbox{supp}(#1)}

\renewcommand{\vec}[1]{\mathit{vec}(#1)}
\newcommand{\qmod}[1]{\mbox{QM}[#1]}
\newcommand{\ideal}[1]{\mbox{Ideal}[#1]}
\newcommand{\st}{\mathit{s.t.}}
\newcommand{\hm}{\mathit{hom}}

\newcommand{\den}{\mathit{den}}

\newcommand{\bdes}{\begin{description}}
	\newcommand{\edes}{\end{description}}
\newcommand{\bal}{\begin{align}}
	\newcommand{\eal}{\end{align}}
\newcommand{\bnum}{\begin{enumerate}}
	\newcommand{\enum}{\end{enumerate}}
\newcommand{\bit}{\begin{itemize}}
	\newcommand{\eit}{\end{itemize}}
\newcommand{\bea}{\begin{eqnarray}}
	\newcommand{\eea}{\end{eqnarray}}
\newcommand{\be}{\begin{equation}}
	\newcommand{\ee}{\end{equation}}
\newcommand{\baray}{\begin{array}}
	\newcommand{\earay}{\end{array}}
\newcommand{\bsry}{\begin{subarray}}
	\newcommand{\esry}{\end{subarray}}
\newcommand{\bca}{\begin{cases}}
	\newcommand{\eca}{\end{cases}}
\newcommand{\bcen}{\begin{center}}
	\newcommand{\ecen}{\end{center}}
\newcommand{\bbm}{\begin{bmatrix}}
	\newcommand{\ebm}{\end{bmatrix}}
\newcommand{\bmx}{\begin{matrix}}
	\newcommand{\emx}{\end{matrix}}
\newcommand{\bpm}{\begin{pmatrix}}
	\newcommand{\epm}{\end{pmatrix}}
\newcommand{\btab}{\begin{tabular}}
	\newcommand{\etab}{\end{tabular}}

\theoremstyle{plain}
\newtheorem{theorem}{Theorem}[section]

\newtheorem{prop}[theorem]{Proposition}

\newtheorem*{claim*}{Claim}
\newtheorem{thm}[theorem]{Theorem}
\theoremstyle{definition}
\newtheorem{example}[theorem]{Example}

\newtheorem{assumption}[theorem]{Assumption}
\newtheorem{remark}[theorem]{Remark}

\newtheorem{conj}[theorem]{Conjecture}

\numberwithin{equation}{section}
\numberwithin{table}{section}

\def\rn{{\mathbb{R}^n}}
\def\r{{\mathbb{R}}}

\def\n{{\mathbb{N}}}

\def\n{{\mathbb{N}}}

\begin{document}

\title[Finite convergence of Moment-SOS relaxations]
{Finite convergence of Moment-SOS relaxations with non-real radical ideals}

\author[Lei Huang]{Lei Huang}
\address{Lei Huang, Department of Mathematics,
	University of California San Diego,
	9500 Gilman Drive, La Jolla, CA, USA, 92093.}
\email{leh010@ucsd.edu}

\author[Jiawang Nie]{Jiawang~Nie}
\address{Jiawang Nie,  Department of Mathematics,
University of California San Diego,
9500 Gilman Drive, La Jolla, CA, USA, 92093.}
\email{njw@math.ucsd.edu}

\author[Ya-Xiang Yuan]{Ya-Xiang Yuan}
\address{Ya-Xiang Yuan,
 Institute of Computational Mathematics and Scientific/Engineering Computing,
 Academy of Mathematics and Systems Science,
 Chinese Academy of Sciences, Beijing, China, 100049.}
\email{yyx@lsec.cc.ac.cn}



\begin{abstract}
We consider the linear conic optimization problem with the cone of nonnegative polynomials.
Its dual optimization problem is the generalized moment problem.
Moment-SOS relaxations are powerful for solving them.
This paper studies finite convergence of the Moment-SOS hierarchy
when the constraining set is defined by equations
whose ideal may not be real radical.
Under the archimedeanness, we show that the Moment-SOS hierarchy
has finite convergence if some classical optimality conditions hold
at every minimizer of the optimal nonnegative polynomial
for the linear conic optimization problem.
When the archimedeanness fails (this is the case for unbounded sets),
we propose a homogenized Moment-SOS hierarchy
and prove its finite convergence under similar assumptions.
Furthermore, we also prove the finite convergence of
the Moment-SOS hierarchy with denominators.
In particular, this paper resolves a conjecture posed in the earlier work.
\end{abstract}

\maketitle

\section{Introduction}

Let $\mathbb{R}[x] \coloneqq \mathbb{R}[x_1,\dots,x_n]$
denote the ring of polynomials in
$x:=(x_1,\dots,x_n)$ with real coefficients,
and let $\mathbb{R}[x]_d$ be the set of polynomials in
$\re[x]$ with degrees at most $d$.
For a nonnegative integer vector power
$\alpha \coloneqq (\alpha_{1},\dots,\alpha_{n})$,
the $\af$th moment of a Borel measure $\mu$ on $\re^n$ is the integral
$\int  x^\af\mt{d} \mu$ for the monomial
$x^\af \coloneqq x_1^{\alpha_{1}} \cdots x_n^{\alpha_{n}}$.
The support of $\mu$, for which we denote $\supp{\mu}$,
is the smallest closed set $T \subseteq \re^n$
such that $\mu(\re^n \backslash T)=0$.
For the degree $d>0$, denote the power set
($|\alpha| \coloneqq \af_1 + \cdots + \af_n$)
\[
\mathbb{N}_{d}^{n}  \coloneqq  \left\{\alpha \in \mathbb{N}^{n} :
|\alpha|   \le  d \right\}.
\]
Let $\re^{ \N^n_{d} }$ denote the space of real vectors $y$
that are labeled by $\af \in \N^n_{d}$, i.e.,
\[
y \, = \, (y_\af)_{ \af \in \N^n_{d} }.
\]
Such a vector $y$ is called a
{\it truncated multi-sequence} (tms) of degree $d$.
For a closed set $K\subseteq \mR^n$,
let $\mathscr{R}_{d}(K)$ denote the moment cone
\be \nonumber
\mathscr{R}_d(K) \, \, \coloneqq \, \,
\Bigg \{  y  \in \re^{ \N^n_{d} }
\left| \baray{c}
\exists \,\, \text{a Borel measure} \, \, \mu, \,\, \supp{\mu} \subseteq  K, \\
y_{\alpha}=\int x^{\alpha} \mt{~d}\mu \, \text{~for each~} \,  \af \in \N^n_{d}
\earay \right.
\Bigg \}.
\ee
The moment cone $\mathscr{R}_d(K)$ is always convex.
It is closed if $K$ is compact,
but it is typically not closed if $K$ is noncompact.
We define the bilinear operation $\langle \cdot, \cdot \rangle$
between $\re[x]_d$ and $\re^{ \N^n_{d} }$:
\be \nonumber
\langle p, y \rangle \,= \, \sum_{ \af \in \N^n_{d} } p_\af y_\af
\quad \mbox{for} \quad
p  = \sum_{ \af \in \N^n_{d} } p_\af x^\af .
\ee

We study the  generalized moment problem (GMP),
which is also called the generalized problem of moments
in some literature (see \cite{Las08})
\be  \label{gpm}
\left\{ \baray{rl}
\min & \langle f, y\rangle  \\
\st & \langle a_i,y\rangle   =  b_i \, \, ( 1 \le i \le m_1),\\
&\langle a_i,y\rangle   \geq  b_i \, \, ( m_1 < i \le m), \\
&   ~y \in \mathscr{R}_{d}(K) .
\earay \right.
\ee
In the above, $f$, $a_1, \ldots, a_m$ are given polynomials in
$\mR[x]_d$, $m_1 \le m$ are integers
and $b_1,\ldots, b_m \in \re$ are given constants.
For convenience, we denote the vector
$b = (b_1, \dots, b_m)$.
The GMP has broad applications in
optimal transport \cite{mona},
distributionally robust optimization \cite{nieyz,NieZhong23},
computational geometry and optimal control \cite{HDLJ},
different ranks for matrices and tensors \cite{DNY22,gls,klm},
and tensor optimization \cite{NZ16,STNN17,NieZhang18}.
More applications can be found in
\cite{HDLJ,Las08,Lau09,niebook}.

The $d$th degree cone of polynomials that are nonnegative on $K$ is
\[
\mathscr{P}_{d}(K) \, \coloneqq \, \{p\in \mR[x]_d:
p(x) \ge 0 \quad \mbox{for all} \,\, x \in K \}.
\]
Note that $\mathscr{P}_{d}(K)$ is a closed convex cone of $\re[x]_d$.
The nonnegative polynomial cone $\mathscr{P}_d(K)$
and the moment cone $\mathscr{R}_d(K)$ are dual to each other,
under the above bilinear operation $\langle \cdot, \cdot \rangle$, i.e.,
\[
\langle p, y \rangle \ge 0 \quad
\text{for all} \,\, p \in \mathscr{P}_d(K)
\,\, \text{and for all} \, \, y \in \mathscr{R}_d(K) .
\]
It can be shown that the dual optimization of \reff{gpm} is
the following linear conic optimization problem:
\be  \label{gpm:dual}
\left\{ \baray{rl}
\max &  b_1\theta_1 + \cdots + b_m \theta_m  \\
\st &  f - \sum\limits_{i=1}^m \theta_ia_i
\in \mathscr{P}_{d}(K),\\
&  \theta_i \geq 0 \,\, (m_1 < i \le m), \\
&  \theta = (\theta_1, \ldots, \theta_m) \in \re^m.
\earay \right.
\ee
We refer to \cite{Las08,niebook}
for more details about the dual relationship.

Moment-SOS relaxations are introduced by Lasserre \cite{Las08}
for solving the pair \reff{gpm}-\reff{gpm:dual}.
Under the archimedeanness, the Moment-SOS hierarchy
has asymptotic convergence. When the set $K$ is a simplex or sphere,
the convergence rates of the  Moment-SOS hierarchy are studied in \cite{kfe}.
We also refer to the survey by De Klerk and Laurent \cite{klerk2019survey}
for the convergence rate analysis.

The classical polynomial optimization problem  (see \cite{Las01})
can be viewed as a special case of the GMP.
There are interesting results about the finite convergence
of the  Moment-SOS hierarchy for solving polynomial optimization.
Under the archimedeanness,  Nie \cite{nieopcd} proved that
the Moment-SOS hierarchy has finite convergence
if some classical optimality conditions
hold at every minimizer
(throughout this paper, a minimizer means a global minimizer,
unless its meaning is otherwise specified).
For convex polynomial optimization, the Moment-SOS hierarchy
has finite convergence under the strict convexity
or SOS-convexity conditions \cite{dKlLau11,Las09}.
When the equality constraints give a finite set,
the Moment-SOS hierarchy also has finite convergence,
as shown in \cite{LLR08,Lau07,Nie13}.
%
%
For  asymptotic convergence rates of the Moment-SOS hierarchy
for solving polynomial optimization, we refer to \cite{bmmp,FF20,maim,niesch,slot1}.

\subsection*{Contributions}

Despite rich work for polynomial optimization,
there are relatively less results about finite convergence
of the Moment-SOS hierarchy for solving generalized moment problems.
A partial reason for this is that the classical optimality conditions
for nonlinear programming \cite{Bert97}
do not generalize conveniently for \reff{gpm}-\reff{gpm:dual}.
Consequently, the finite convergence theory in \cite{nieopcd}
does not generalize directly.

Assume that $K$ is the semialgebraic set given as
\[ 
K \coloneqq
\left\{x \in \mathbb{R}^{n} \left| \begin{array}{l}
c_{j}(x)=0~(j \in \mathcal{E}), \\
c_{j}(x) \geq 0~(j \in \mathcal{I})
\end{array}\right\},\right.
\]
where $ c_j$ $(j\in \mc{E} \cup \mc{I})$ are given polynomials.
The $\mc{E}$ and $\mc{I}$ are disjoint labeling sets
for equality and inequality constraining polynomials respectively.
Suppose $\theta^* = (\theta^*_1, \ldots, \theta^*_m)$
is a maximizer of \reff{gpm:dual}.
Consider the polynomial optimization problem
\be  \label{1:stand:opt}
\left\{ \baray{rl}
\min & f_{\theta^*}(x) \, \coloneqq \,
 f(x) - \sum\limits_{i=1}^m \theta_i^* a_i(x)  \\
\st &  	c_{j}(x)=0~(j \in \mathcal{E}), \\
&c_{j}(x) \geq 0~(j \in \mathcal{I}).
\earay \right.
\ee
It is worthy to note that the polynomial $f_{\theta^*}(x)$ is nonnegative on $K$ and it also typically has a zero $u\in K$.
For convenience, we denote the polynomial tuples
\[
c_{eq}=(c_j(x))_{j\in \mathcal{E}}, ~
c_{in}=(c_j(x))_{j\in \mathcal{I}}.
\]
When $\ideal{c_{eq}}+\qmod{c_{in}}$ is archimedean
(see Section~\ref{ssc:pre:pop} for the notation),
if the linear independence constraint qualification,
strict complementarity and second order sufficient conditions
hold at each minimizer of \reff{1:stand:opt}, then there exists a polynomial
$\sig \in \qmod{c_{in}}$ such that
$f_{\theta^*} - \sig  \equiv 0$ on $V_{\re}(c_{eq})$.
If $\ideal{c_{eq}}$ is a real radical ideal
(see Section~\ref{ssc:pre:pop} for the definition),
then $f_{\theta^*}(x) - \sig \in \ideal{ c_{eq} }$,
or equivalently, $f_{\theta^*}(x) \in \ideal{c_{eq}}+\qmod{c_{in}}$.
However, if $\ideal{c_{eq}}$ is not real radical,
we typically do not have $f_{\theta^*}(x) \in \ideal{c_{eq}}+\qmod{c_{in}}$.
This is a major difficulty for studying finite convergence
of the Moment-SOS hierarchy when solving generalized moment problems.

We remark that the equality constraint $c_j(x)=0$ in \reff{1:stand:opt}
can be equivalently replaced by two inequalities
$c_j(x)\geq 0$, $-c_j(x)\geq0$. However, by doing this,
the linear independence constraint qualification
will fail for this reformulation and forbid us
to get the finite convergence results.

In this paper, we study the Moment-SOS hierarchy for solving \reff{gpm}-\reff{gpm:dual}
{\em without} the assumption that $\ideal{c_{eq}}$ is real radical.
Indeed, we can prove its finite convergence under the archimedeanness
and the above mentioned optimality conditions.
Our major contributions are as follows.

\bit

\item
First, we consider the case that $\ideal{c_{eq}}+\qmod{c_{in}}$ is archimedean
but $\ideal{c_{eq}}$ is not necessarily real radical.
We prove that the Moment-SOS hierarchy for solving
\reff{gpm}-\reff{gpm:dual} has finite convergence
if the linear independence constraint qualification,
strict complementarity and second order sufficient conditions hold
at every minimizer of \reff{1:stand:opt}. In particular,  we are also able to
show that every minimizer of the moment relaxation for \reff{gpm}
must satisfy the flat truncation, for all sufficient high relaxation orders.
This also gives a certificate for checking finite convergence
of the Moment-SOS hierarchy in computational practice. Applications in polynomial optimization and the semialgebraic super resolution problem are also discussed.

\item
Second, we consider the case that $\ideal{c_{eq}}+\qmod{c_{in}}$
is not archimedean. This is typically the case when $K$ is an unbounded set
and the Moment-SOS hierarchy generally does not converge.
For this case, we propose a homogenized Moment-SOS hierarchy
introduced in \cite{hny, hny1} and prove its finite convergence
under similar optimality conditions for the homogenization of
the optimization \reff{1:stand:opt}.
Consequently, we are also able to show the finite convergence of
the Moment-SOS hierarchy with denominators,
motivated by the Putinar-Vasilescu type Positivstellensatz
\cite{MLM21,putinar1999positive}.
In particular, our results resolve Conjecture~8.2 in the earlier work \cite{hny}.

\eit
	
This paper is organized as follows. Section~\ref{sc:pre}
reviews some backgrounds for polynomial optimization and moment problems.
In Section~\ref{sc:finite}, we prove the finite convergence of the Moment-SOS hierarchy
under the archimedeanness and  optimality conditions,
without the assumption of real radicalness. 
In Section~\ref{sc:homo}, we prove the finite convergence of
the homogenized Moment-SOS hierarchy and
the one with denominators. In Section~\ref{sc:su}, we study the semialgebraic super resolution problem.
We make some conclusions and pose a new conjecture
in Section~\ref{sc:con}.

\section{Preliminaries}
\label{sc:pre}

\subsection*{Notation}

The symbol $\mathbb{N}$ (resp., $\mathbb{R}$)
denotes the set of nonnegative integers (resp., real numbers).
For $x = (x_1,\dots,x_n)$ and $\af = (\af_1, \ldots, \af_n)$, denote
\[
x^{\alpha} \coloneqq x_1^{\alpha_{1}}\cdots x_n^{\alpha_{n}}, \quad
|\alpha| \coloneqq  \alpha_{1}+\cdots+\alpha_{n}.
\]
For $t \in \mathbb{R}$, $\lceil t \rceil$
denotes the smallest integer greater than or equal to $t$.
For a positive integer $m$, denote $[m] \coloneqq \{1,\dots,m\}$.
For a matrix $A$, $A^{\mathrm{T}}$ denotes its transpose.
A symmetric matrix $X \succeq 0$  if $X$ is positive semidefinite.
For a set $S \subseteq \rn$, $cl(S)$ and $int(S)$
denote its closure and interior  in the Euclidean topology.
Denote by $\deg(p)$  the total degree of  the polynomial $p$.
For a degree $d$, let $[x]_d$ denote the vector of all monomials in $x$
with degrees  $\leq d$, ordered in the graded alphabetical ordering, i.e.,
\[
[x]_d^{\mathrm{T}} \,=\, [1,  x_1, x_2, \ldots, x_1^2, x_1x_2, \ldots,
x_1^d, x_1^{d-1}x_2, \ldots, x_n^d ].
\]
For a polynomial $p$, $p^\hm$ denotes its homogeneous terms of
the highest degree and $\tilde{p}$ denotes its homogenization, i.e.,
$\tilde{p}(\tilde{x})=x_0^{\deg(p)} p(x/x_0)$
for $\tilde{x} \coloneqq (x_0,x_1,\dots,x_n)$. A homogeneous polynomial $p$
is said to be a form and $p$ is positive definite if $p(x)>0$ for all nonzero $x \in \re^n$.

\subsection{Optimality conditions}
\label{ssc:opcd}

We first review some classical optimality conditions in nonlinear programming \cite{Bert97}.
Consider the nonlinear optimization problem
\be   \label{intro:smo}
\left\{\baray{rl}
\min & f(x) \\
\st  & c_{j}(x)=0~(j \in \mathcal{E}), \\
& c_{j}(x) \geq 0~(j \in \mathcal{I}),
\earay \right.
\ee
where $f,~  c_j$ $(j \in \mathcal{E} \cup \mathcal{I})$ are second order continuous differentiable.
Suppose $x^*$ is a local minimizer of $(\ref{intro:smo})$.
Denote the label set of active constraints at $x^*$
\be \label{nota:J(x*)}
J(x^*) \coloneqq \left\{j \in \mathcal{E} \cup \mathcal{I} :  c_{j}(x^*)=0\right\}.
\ee
If the gradient vector set $\{ \nabla c_{j}(x^*)\}_{j \in J(x^*)}$
 is linearly independent, the linear independence constraint qualification condition (LICQC)
is said to hold at $x^*$. When the LICQC holds at $x^*$,
there exist Lagrange multipliers $\lambda_j$  $(j \in \mathcal{E} \cup \mathcal{I})$ such that
\be  \label{1.1:KKT}
\boxed{
\begin{array}{c}
\nabla f(x^*) = \sum\limits_{j \in \mathcal{E} \cup \mathcal{I}} \lambda_{j} \nabla c_{j}(x^*), \\
\lambda_{j} \geq 0, \, \lambda_{j} c_{j}(x^*)=0~~(j \in \mathcal{I})
\end{array}
}.
\ee
Moreover, if $\lambda_j + c_j(x^*) >0$ for every $j \in \mathcal{I}$,
then the strict complementarity condition (SCC) is said to hold at $x^*$.
For the above Lagrange multipliers, the Lagrange function is
\[
\mathscr{L}(x) \,  \coloneqq  \, f(x)-\sum_{j \in \mathcal{E} \cup \mathcal{I}} \lambda_{j} c_{j}(x) .
\]
Under the LICQC, the second order necessary condition (SONC)
must hold at $x^*$, i.e.,
the following holds ($\nabla^2$ denotes the Hessian operator)
\be \label{opcd:SONC}
v^{\mathrm{T}}\nabla^2 \big( \mathscr{L}(x^*) \big) v \geq 0
\ee
for every $v\in \mathbb{R}^n$ satisfying
\be  \label{tangent}
v^{\mathrm{T}}\nabla c_{j}(x^*)=0 ~\text{~for all~}~ j \in J(x^*).
\ee
The second order sufficient condition (SOSC) is said to hold at $x^*$ if
\be \label{opcd:SOSC}
v^{\mathrm{T}}\nabla^2 \big( \mathscr{L}(x^*) \big) v > 0,
\ee
for every nonzero $v$ satisfying \reff{tangent}.

\subsection{Some basics for polynomial optimization}
\label{ssc:pre:pop}

Here we review some basics for polynomial optimization and moment problems.
A subset $I  \subseteq \mathbb{R}[x]$ is called an ideal of $\re[x]$
if $I \cdot \mathbb{R}[x] \subseteq \mathbb{R}[x]$, $I+I \subseteq I$.
For a polynomial tuple $h  \coloneqq (h_1,\dots, h_t)$,
$\ideal{h}$ denotes the ideal generated by $h$, i.e.,
\begin{equation*}
\ideal{h} \,= \, h_1 \cdot \mathbb{R}[x]+\cdots+h_t \cdot \mathbb{R}[x].
\end{equation*}
For a degree $k$, the $k$th degree truncation of $\ideal{h}$ is
\begin{equation*}
\ideal{h}_{k} = h_1 \cdot \mathbb{R}[x]_{k-\deg(h_1)}+\cdots+
h_t \cdot \mathbb{R}[x]_{k-\deg(h_t)}.
\end{equation*}
Its  real variety is
\begin{equation*}
V_{\mR}(h)  \, \coloneqq \,  \{x \in \mR^n \mid h_1(x)=\cdots=h_t(x)=0\}.
\end{equation*}
For a set $T \subseteq \re^n$, its vanishing ideal is
\[
\ideal{T}  = \{ p \in \re[x]: \, p(u) = 0
\, \mbox{~for all~} \, u \in T \}.
\]
Clearly, if $T = V_{\mR}(h)$, then $\ideal{h} \subseteq \ideal{T}$.
However, the reverse inclusion  may not hold.
The ideal $\ideal{h}$ is said to be {\it real radical}
if $\ideal{h}=\ideal{T}$ for $T = V_{\mR}(h)$.

A polynomial $p$ is said to be a sum of squares (SOS) if
$p=p_1^2+\dots+p_s^2$ for $p_1,\dots,p_s \in \mathbb{R}[x]$.
The set of all SOS polynomials in $x$ is denoted as $\Sigma[x]$.
For a degree $k$, denote the truncation
\[
\Sigma[x]_{k} \, \coloneqq  \, \Sigma[x] \cap  \mathbb{R}[x]_{k} .
\]
The notion of SOS polynomials plays a core role in optimization.
Interesting, this notion can also be defined for matrix polynomials
(see \cite{HilNie08,NiePMI}).
Clearly, an SOS polynomial is nonnegative everywhere,
while the reverse is typically not true.
The quality of SOS approximations
is discussed in \cite{FF20,NieSOSbd}.
For a polynomial tuple $g=(g_1,\dots,g_{\ell})$,
the  quadratic module generated by $g$ is
\be
\qmod{g} \,  \coloneqq  \,  \Sigma[x]+ g_1 \cdot \Sigma[x]+\cdots+ g_{\ell} \cdot \Sigma[x].
\ee
Similarly, the $k$th degree truncation of $\qmod{g}$ is
\be
\qmod{g}_{k} = \Sigma[x]_{k}+ g_1 \cdot \Sigma[x]_{k-\deg(g_1)}+\cdots
+ g_{\ell} \cdot \Sigma[x]_{k-\deg(g_{\ell})}.
\ee

The sum $\ideal{h}+\qmod{g}$ is said to be archimedean if there exists
$R >0$ such that $R-\|x\|^2 \in \ideal{h}+\qmod{g}$.
If it is archimedean, then the set
\[
S \coloneqq \left\{x \in \mathbb{R}^{n} \mid h(x)=0, g(x) \geq 0\right\}
\]
must be compact. Clearly, if $p \in \ideal{h}+\qmod{g}$, then $p \ge 0$ on $S$
while the converse is not always true. However,
if $p$ is positive on $S$ and $\ideal{h}+\qmod{g}$ is archimedean,
we have $p \in \ideal{h}+\qmod{g}$.
This conclusion is often referred as Putinar's Positivstellensatz.

\begin{thm}[\cite{putinar1993positive}] \label{thm2.1}
 Suppose  $\ideal{h}+\qmod{g}$ is  archimedean.
If a polynomial $p>0$ on  $S$, then $p \in \ideal{h}+\qmod{g}$.
\end{thm}

In the above theorem, if  $p \ge 0$ on  $S$, $\ideal{h}$ is real radical and the
LICQC, SCC and SOSC hold at each zero of $p$ on $S$,
then we also have $p \in \ideal{h}+\qmod{g}$.
This is shown in \cite{nieopcd}.

For a polynomial $q\in \mR[x]_{2k}$ and a tms $w\in\re^{\N_{2k}^n}$,
the $k$th order {\em localizing matrix} of   $q$ for $w$
is the symmetric matrix $L_{q}^{(k)}[w]$ such that
\be \label{locmat:gi}
\vec{p}^T\Big(L_{q}^{(k)}[w]\Big)\vec{p}\,=\,\langle q p^2, w\rangle,
\ee
for  every polynomial $p \in \re[x]_s$ with $2s \le 2k-\deg(q)$.
In the above, $\vec{p}$ denotes the coefficient vector of $p$.
In particular, for the constant one polynomial $q = 1$,
the localizing matrix $L_{q}^{(k)}[w]$ becomes the kth order {\it moment matrix}
$
M_k[w]\,\coloneqq\, L_1^{(k)}[w].
$
For instance, when $n=2$, we have
\[
M_2[w] = \begin{bmatrix}
w_{00} & w_{10} & w_{01} & w_{20} & w_{11} & w_{02}\\
w_{10} & w_{20} & w_{11} & w_{30} & w_{21} & w_{12}\\
w_{01} & w_{11} & w_{02} & w_{21} & w_{12} & w_{03}\\
w_{20} & w_{30} & w_{21} & w_{40} & w_{31} & w_{22}\\
w_{11} & w_{21} & w_{12} & w_{31} & w_{22} & w_{13}\\
w_{02} & w_{12} & w_{03} & w_{22} & w_{13} & w_{04}
\end{bmatrix},
\]
\[
L_{x_1-x_2^2}^{(2)}[w] = \begin{bmatrix}
w_{10}-w_{02} & w_{20}-w_{12} & w_{11}-w_{03}\\
w_{20}-w_{12} & w_{30}-w_{22} & w_{21}-w_{13}\\
w_{11}-w_{03} & w_{21}-w_{13} & w_{12}-w_{04}
\end{bmatrix}.
\]
For the polynomial $q$, let $\mathscr{V}_{q}^{(2k)}[z]$
denote the vector such that
\be  \label{locvec:Vci}
\langle q  p, w \rangle  \, = \,
\big( \mathscr{V}_{q}^{(2k)}[w] \big)^T \vec{p},
\ee
for  every polynomial $p \in \re[x]_s$ with $s \le 2k-\deg(q)$.
The $\mathscr{V}_{q}^{(2k)}[w]$  is called the
{\it localizing vector} of $q$ and $w$.
We refer to \cite[Chap.~2]{niebook} for
more details about quadratic modules and localizing matrices/vectors.
They are basic tools for solving polynomial optimization and moment problems.
We refer to the books and surveys
\cite{HDLJ,LasBk15,Lau09,marshall2008positive,Sch09}
for more comprehensive introductions.

\section{Finite convergence of the  Moment-SOS hierarchy}
\label{sc:finite}

In this section, we study the Moment-SOS hierarchy
for solving the optimization problems \reff{gpm}-\reff{gpm:dual}.
Assume $K$ is the basic closed semialgebraic set
\be  \label{feas:set}
K \coloneqq
\left\{x \in \mathbb{R}^{n} \left| \begin{array}{l}
c_{j}(x)=0~(j \in \mathcal{E}), \\
c_{j}(x) \geq 0~(j \in \mathcal{I})
\end{array}\right\},\right.
\ee
where $c_j$ $(j \in \mathcal{E}\cup \mathcal{I})$ are given polynomials in $x$.
The $\mc{E}$ and $\mc{I}$ are disjoint labeling sets.
For convenience, we denote the polynomial tuples
\[
c_{eq}=(c_j(x))_{j\in \mathcal{E}},~ c_{in}=(c_j(x))_{j\in \mathcal{I}}.
\]
This section studies the finite convergence theory for the Moment-SOS hierarchy,
without the assumption that $\ideal{c_{eq}}$ is real radical.

\subsection{The Moment-SOS hierarchy}

Denote the degrees
\[
d_K = \max\limits_{j \in \mathcal{E}\cup \mathcal{I}}
\big \{ \lceil \deg(c_j)/2 \rceil  \big\}, \quad
d_0 =\max \Big \{\lceil \deg(f)/2 \rceil, ~d_K, \,
\max\limits_{1 \le i \le m }  \lceil \deg(a_i) /2\rceil \Big \}.
\]
For a degree $k \geq  d_0$, the $k$th moment relaxation for   \reff{gpm} is
\be  \label{momre:gpm}
\left\{ \baray{cl}
\min &  \langle f, w \rangle  \\
\st  &  \langle a_i,w\rangle  =  b_i \, \, ( 1 \le  i \le m_1),  \\
&\langle a_i,w\rangle   \geq  b_i \, \, (m_1 < i \le m),\\
& \mathscr{V}_{c_{j}}^{(2k)}[w] = 0~ (j\in \mathcal{E}), \\
&L_{c_{j}}^{(k)}[w] \succeq 0~(j\in \mathcal{I}), \\
& M_k[w] \succeq 0, \, w \in \mathbb{R}^{\mathbb{N}_{2 k}^{n}} .
\earay \right.
\ee
The dual optimization of \reff{momre:gpm} is the $k$th order SOS relaxation for  \reff{gpm:dual}:
\be \label{sos:gpm}
\left\{ \baray{cl}
\max &  b_1 \theta_1 + \cdots + b_m \theta_m  \\
\st & f - \sum\limits_{i=1}^m \theta_ia_i
\in  \ideal{c_{eq}}_{2k}+\qmod{c_{in}}_{2k},\\
& \theta \in \mR^m ,~ \theta_{m_1+1} \geq 0,\ldots, \theta_{m} \geq 0.
\earay \right.
\ee
As $k$ goes to infinity, the sequence of \reff{momre:gpm}-\reff{sos:gpm}
is called the Moment-SOS hierarchy for solving \reff{gpm}-\reff{gpm:dual}.

Let   $\phi^*$, $\vartheta^*$ denote the optimal value of
 \reff{gpm}, \reff{gpm:dual} respectively,
and let  $\phi_k$, $\vartheta_k$ denote the optimal value of
 \reff{momre:gpm}, \reff{sos:gpm} respectively.
By the weak duality, it holds that
\[
\vartheta_k \leq  \vartheta^*, \quad
\phi_k \le \phi^*, \quad
\vartheta_k \leq \phi_k, \quad
\vartheta^* \leq \phi^*.
\]
The hierarchy of \reff{momre:gpm}--\reff{sos:gpm} is said to have finite convergence
(or be {\em tight}) if there exists $k_0\in \mathbb{N}$ such that
 $\vartheta_k= \phi_k = \vartheta^* = \phi^*$ for all $k\geq k_0$.

In computational practice, a convenient criterion to check finite convergence
is the flat truncation \cite{CF05,nie2013certifying}.
Suppose  $w^*\in \r^{\n^n_{2k}}$  is a minimizer of $(\ref{momre:gpm})$.
The tms $w^*$ is said to have a flat truncation if there exists a degree
$t \in [d_0, k]$ such that
\begin{equation}  	\label{flat1.3}
\operatorname{rank} M_{t-d_{K}}[w^*|_{2t}]
\, = \,  \operatorname{rank} M_{t}[w^*|_{2t}].
\end{equation}
When it holds,  we must have
$ \phi_k =\phi^*$
and $w^*|_{d}$ is the optimal solution to \reff{gpm}.
We refer to \cite{HDLJ,Lau05, Lau09,niebook}
for more details.

\subsection{The finite convergence theory}

We now prove the finite convergence of the Moment-SOS hierarchy
\reff{momre:gpm}--\reff{sos:gpm}  under the archimedeanness and  optimality conditions.
This extends the finite convergence results in \cite{nieopcd}
for polynomial optimization to the generalized moment problem.

For convenience of notation, we denote the polynomial
\[
f_{\theta}(x) \, \coloneqq \, f(x) - \sum\limits_{i=1}^m \theta_i a_i(x).
\]
Suppose $\theta^* \in \mR^m$ is a maximizer of \reff{gpm:dual}.
Consider the polynomial optimization problem
\be  \label{stand:opt}
\left\{ \baray{rl}
\min & f_{\theta^*}(x)   \\
\st &  	c_{j}(x)=0~(j \in \mathcal{E}), \\
& c_{j}(x) \geq 0~(j \in \mathcal{I}).
\earay \right.
\ee
Generally, the optimal value of
\reff{stand:opt} is zero and each of its minimizers
is a zero of $f_{\theta^*}(x)$ in the set $K$.
This is guaranteed under the following assumption.

\begin{assumption} \label{assum}
Suppose $\theta^* \in \mR^m$ is a maximizer of \reff{gpm:dual},
$y^*$ is a minimizer of \reff{gpm} with $y^* \neq 0$,
and the Slater condition holds for \reff{gpm:dual}, i.e., there exits a feasible point
$\bar{\theta}\in \mR^m$ of \reff{gpm:dual}  such that
$f_{ \bar{\theta} } \in  int(\mathscr{P}_d(K))$.
\end{assumption}

Under Assumption \ref{assum}, the strong duality holds between
\reff{gpm:dual} and \reff{gpm}, so
\[
0 = \langle f, y^* \rangle - b^{\mathrm{T}} \theta^*
= \langle f_{\theta^*}, y^* \rangle = \int f_{\theta^*}(x)  \mt{d} \mu ,
\]
for every representing measure $\mu$ of $y^*$, with $\supp{\mu} \subseteq K$.
Note that $f_{\theta^*}\geq 0$ on $K$.
If $y^* \neq 0$, then $\supp{\mu} \ne \emptyset$
and each point in $\supp{\mu}$ must be a zero of $f_{\theta^*}$ in the set $K$.
The following is our major conclusion in this paper.

\begin{thm} \label{finite:gpm}
Suppose $\ideal{c_{eq}}+\qmod{c_{in}}$ is archimedean and
Assumption~\ref{assum} holds.
If the LICQC, SCC and SOSC hold at every  minimizer of \reff{stand:opt},
then we have:

\bit

\item [(i)]
The hierarchy of \reff{momre:gpm}--\reff{sos:gpm} has finite convergence, i.e.,
$\vartheta_k = \phi_k = \vartheta^* = \phi^*$ for all $k$ big enough.

\item [(ii)]
Every minimizer of the moment relaxation \reff{momre:gpm}
must have a flat truncation, when $k$ is sufficiently large.

\eit
\end{thm}

\begin{proof}
(i)	Note that the minimum value of \reff{stand:opt} is nonnegative.
Since the LICQC, SCC and SOSC hold at every minimizer of \reff{stand:opt},
by Theorem 1.1 of \cite{nieopcd}, there exists $\sigma_0 \in \qmod{c_{in}}$ such that
\[
f_{\theta^*}  \equiv \sigma_0 \mod \ideal{V_{\mR}(c_{eq})}.
\]
This means that  $\hat{f} \coloneqq f_{\theta^*} - \sigma_0$
vanishes identically on the real variety $V_{\mR}(c_{eq})$.
By the Real Nullstellensatz (see \cite{niebook}), there exist
$k_1\in \N$, $\sigma_1\in \Sigma[x]$ such that
$\hat{f}^{2k_1}+\sigma_1  \in	\ideal{c_{eq}}$.
The archimedeanness of $\ideal{c_{eq}}+\qmod{c_{in}}$
implies $K$ is compact. By the Slater condition,
there exist a feasible point $\bar{\theta}\in \mR^m$ for \reff{gpm:dual}
and a scalar $\dt >0$ such that
$f_{\bar{\theta}} \ge \dt >0$ on $K$.
By Theorem \ref{thm2.1}, we have that
\be \label{ell-c}
f_{\bar{\theta}}-  \dt \in \ideal{c_{eq}}+\qmod{c_{in}}.
\ee
%
For $\gamma\geq \frac{1}{2k_1}(1-\frac{1}{2k_1})^{2k_1-1} $,
the polynomial $\omega(t) \coloneqq 1+t+ \gamma t^{2 k_1}$
is SOS (cf. \cite{Nie13}), so
\be \label{cwtc}
\dt \cdot\omega(t/\dt) ~ = ~  \dt + t +
\frac{1}{2k_1}(1-\frac{1}{2k_1})^{2k_1-1} \dt^{1-2k_1} t^{2 k_1} \in \Sigma[t]_{2k_1}.
\ee
Combining with \reff{ell-c} and \reff{cwtc}, we have
\be  \label{stx:sos}
s(t,x)   \coloneqq  f_{\bar{\theta}} - \dt  + (\dt+t+\eta t^{2k_1})
\in \ideal{c_{eq}}+\qmod{c_{in}}+\Sigma[t],
\ee
for every scalar
\[
 \eta  \, \ge \, \frac{1}{2k_1}(1-\frac{1}{2k_1})^{2k_1-1} \dt^{1-2k_1} .
\]
For $0<\lambda\leq 1$, replacing $t$ by $(\lmd^{-1}-1) \hat{f}$
in $s(t,x)$, we then get
\[
\lambda f_{\bar{\theta}}+(1-\lambda) \hat{f}  =  \lambda
s( (\lmd^{-1}-1) \hat{f},x) - \lambda^{1-2k_1}(1-\lambda)^{2k_1}\eta\hat{f}^{2k_1}.
\]
Note that
\begin{equation*}
	\baray{cc}
	\lambda 	 f_{\bar{\theta}}+(1-\lambda) \hat{f}&=\lambda 	f_{\bar{\theta}}
     +(1-\lambda) f_{\theta^*}-(1-\lambda)\sigma_0\\
	&=f_{(1-\lambda)\theta^*+\lambda\bar{\theta}}-(1-\lambda)\sigma_0.
	\earay
\end{equation*}
Thus, we get
\be \nonumber
f_{(1-\lambda)\theta^*+\lambda\bar{\theta}} = (1-\lambda)\sigma_0 +
 \lambda s( (\lmd^{-1} -1 ) \hat{f},x)-\lambda^{1-2k_1}(1-\lambda)^{2k_1}\eta \hat{f}^{2k_1}.
\ee
We select $\eta$ sufficiently large so that \reff{stx:sos} holds. Then
there exists $k_0\in \mathbb{N}$ such that
\be \label{lambda:ind}
f_{(1-\lambda)\theta^*+\lambda\bar{\theta}} \in
    \ideal{c_{eq}}_{2k_0}+\qmod{c_{in}}_{2k_0}.
\ee
Since \reff{lambda:ind} holds for all  $0<\lambda\leq 1$
and $k_0$ is independent of $\lambda$, we know
$(1-\lambda)\theta^*+\lambda\bar{\theta}$   is feasible for
\reff{sos:gpm} at the relaxation order $k_0$, for all $0<\lambda\leq 1$.
As $\lambda \rightarrow 0$, we get
\be
  b^{\mathrm{T}} \big( (1-\lambda)\theta^*+\lambda\bar{\theta} \big )
= (1-\lambda)  b^{\mathrm{T}}\theta^* + \lambda  b^{\mathrm{T}} \bar{\theta}
 \rightarrow   b^{\mathrm{T}} \theta^* .
\ee
This implies that $\vartheta_{k_0} =  b^{\mathrm{T}} \theta^* = \vartheta^*$.
Since the Slater condition holds for \reff{gpm:dual}, we know
$\vartheta^*=\phi^*=\vartheta_{k_0} = \phi_{k_0} $, by the strong duality.
Hence, the hierarchy \reff{momre:gpm}--\reff{sos:gpm} has finite convergence.

(ii)
Suppose $w^{(k)}$ is a minimizer of \reff{momre:gpm} at the relaxation order $k$. First, we prove that $w^{(k)}_0\neq 0$. Suppose otherwise $w^{(k)}_0=0$. The constraint $M_k[w^{(k)}] \succeq 0$ implies that $w^{(k)}_{\alpha}=0$ for all $|\alpha| \leq k$. By Lemma~5.7 \cite{Lau09}, we have that $M_k[w^{(k)}]vec(x^{\alpha})=0$ for all $\alpha\leq k-1$.
	For all $|\alpha| \leq 2 k-2$, we can write $\alpha=\beta+\eta$ with $|\beta|,|\eta| \leq k-1$ and the following holds
	$$
	w^{(k) }_\alpha=\operatorname{vec}(x^\beta)^\mathrm{T} M_k[w^{(k)}] \operatorname{vec}(x^\eta)=0 .
	$$
	Clearly, the truncation $\left.w^{( k)}\right|_{2 k-2}$ is flat. It implies that the minimizer  of  \reff{gpm} is  $y^*= 0$, which is a contradiction.

Thus, we know that $w^{(k)}_0>0$. Up to scaling, we can assume $w^{(k)}_0=1$.
Since the Slater condition holds for \reff{gpm:dual},
we have  $\vartheta^*=\phi^*$ and \reff{gpm} has a minimizer, saying, $y^*$. Suppose $\mu$ is a $K$-representing measure of $y^*$. Since $y^*\neq 0$, we know that $\supp{\mu} \neq \emptyset$.  It holds that
\be \nonumber
\baray{cl}
\int f_{\theta^*} \mathrm{~d\mu}=\langle f - \sum\limits_{i=1}^m \theta_i^*a_i ,  y^*\rangle&=\langle f,y^*\rangle-\sum\limits_{i=1}^{m_1} \theta_i^*b_i-\sum\limits_{i=m_1+1}^m \theta_i^*\langle  a_i ,  y^* \rangle\\
&\leq \langle f,y^*\rangle-\sum\limits_{i=1}^{m} \theta_i^*b_i\\
&=0.

\earay
\ee
Note that $ f_{\theta^*}(x)\geq 0$ on $K$ and  $\supp{\mu} \neq \emptyset$. Hence, the optimal value of \reff{stand:opt} is 0.   The moment relaxation for polynomial optimization \reff{stand:opt} at the order $k$   is
\be  \label{mom:opt:stan}
\left\{ \baray{cl}
\min &  \langle f - \sum\limits_{i=1}^m \theta_i^*a_i , w \rangle  \\
\st  &  \langle 1,w\rangle  =  1,  \\
&\mathscr{V}_{c_{j}}^{(2k)}[w] =  0~ (j\in \mathcal{E}), \\
&L_{c_{j}}^{(k)}[w] \succeq 0~(j\in \mathcal{I}), \\
& M_k[w] \succeq 0, \, w \in \mathbb{R}^{\mathbb{N}_{2 k}^{n}} .
\earay \right.
\ee
By item (i), we know that
$\vartheta^*=\phi^*=\vartheta_{k}=\phi_{k} $ for $k\geq k_0$ and the following holds
\be \nonumber
\baray{cl}
\langle f_{\theta^*} ,  w^{(k)}\rangle &=\langle f, w^{(k)} \rangle - \sum\limits_{i=1}^{m_1} \theta_i^*b_i-\sum\limits_{i=m_1+1}^m \theta_i^*\langle  a_i ,  w^{(k)}\rangle\\
&\leq \langle f, w^{(k)} \rangle - \sum\limits_{i=1}^m \theta_i^*b_i=0.

\earay
\ee
Since
$
f_{\theta^*} \in cl(\ideal{c_{eq}}_{2k_0}+\qmod{c_{in}}_{2k_0}),
$
we have $\langle f_{\theta^*} ,  w^{(k)}\rangle \geq 0$. Thus,
\[
\langle f_{\theta^*} ,  w^{(k)}\rangle = 0
\]
and $w^{(k)}$ is a minimizer of \reff{mom:opt:stan} for all $k\geq k_0$. Let
\[
Q \, \coloneqq \,  \operatorname{Ideal}\left[c_{e q},
      f_{\theta^*}\right]+\mathrm{QM}\left[c_{i n}\right].
\]
Note that $Q$ is  archimedean and the intersection $J:=Q \cap-Q$
is an ideal. Since the LICQC, SCC, and SOSC hold at every  minimizer of \reff{stand:opt},
the boundary Hessian conditions hold at each zero of
$ f_{\theta^*}$ in $K$ (see \cite{nieopcd}).
By Theorem 2.3 of \cite{mar06},
the coordinate ring $\frac{\mathbb{R}[x]}{J}$ has dimension 0.  Let
$$
S  \coloneqq  \left\{\begin{array}{ll}
x \in \mathbb{R}^n \left| \begin{array}{l}
f_{\theta^*}(x)=0, \\
c_j(x)=0~(j \in \mathcal{E}), \\
c_j(x) \geq 0~(j \in \mathcal{I})
\end{array}\right.
\end{array}\right\} .
$$
Clearly, $S$ is the set of all   minimizers of  \reff{stand:opt}.
The SOSC implies that each minimizer of \reff{stand:opt} is an isolated minimizer.
Note that $K$ is compact. Hence, the set $S$ is finite and
the vanishing ideal $I(S)$ is zero dimensional.
Let $\left\{h_1, \ldots, h_r\right\}$ be a Gr\"{o}bner basis of $I(S)$
with respect to a total degree ordering.
Since $h_t \equiv 0$ $(t=1,\dots,r)$ on $S$ and $J$ is zero dimensional,
by \cite[Corollary~7.4.2]{marshall2008positive},
there exist $\ell\in \mathbb{N}$,
$\phi_0, \phi_j \in \mathbb{R}[x]$, and $\psi_0, \psi_j \in \Sigma[x]$ such that
$$
h_t^{2 \ell}+\sum_{j \in \mathcal{E}} \phi_j c_j+\phi_0 f_{\theta^*}+
\sum_{j \in \mathcal{I}} \psi_j c_j+\psi_0=0 .
$$
When $2 k$ is bigger than the degrees of all above polynomials, we have
\be \label{realzero}
\langle h_t^{2 \ell}, w^{(k)}\rangle+\sum_{j \in \mathcal{E}}\langle\phi_j c_j, w^{(k)}\rangle+\langle\phi_0 f_{\theta^*}, w^{(k)}\rangle+\sum_{j \in \mathcal{I}}\langle\psi_j c_j, w^{(k)}\rangle+\langle\psi_0, w^{(k)}\rangle=0.
\ee
From item (i), we know that  $\hat{f}^{2k_1}+\sigma_1  \in	\ideal{c_{eq}}$ and  the following holds
 \be \nonumber
f_{\theta^*}+\epsilon = ~\sigma_0+ \epsilon \cdot \omega(\hat{f}/\epsilon)-\gamma \epsilon^{1-2k_1}\hat{f}^{2k_1}.
\ee
Thus, when $k_0$ is sufficiently large, we have that for all $\epsilon>0$,
\[
f_{\theta^*}+\epsilon
\in  \ideal{c_{eq}}_{2k_0}+\qmod{c_{in}}_{2k_0}.
\]
Equivalently, there exist $\phi_j^{\epsilon}\in \mR[x]$,
and $\psi_0^\epsilon$, $\psi_j ^\epsilon \in \Sigma[x]_{2k_0}$ such that
\[
\deg(\phi_j^{\epsilon}c_j)\leq 2k_0 ~~(j\in \mathcal{E} ), \quad
\deg(\psi_j^{\epsilon}c_j)\leq 2k_0 ~~(j\in \mathcal{I} ),
\]
\be \label{per}
f_{\theta^*}+\epsilon=\sum_{j \in \mathcal{E}} \phi_j^{\epsilon} c_j+\sum_{j \in \mathcal{I}} \psi_j^{\epsilon} c_j+\psi_0^{\epsilon} .
\ee
Applying the bi-linear operation $\langle \cdot, w^{(k)}\rangle$
to \reff{per} for $k> k_0$, we get
\be
\epsilon=\sum_{j \in \mathcal{E}} \langle \phi_j^{\epsilon}c_j, w^{(k)}\rangle +\sum_{j \in \mathcal{I}} \langle \psi_j^{\epsilon} c_j, w^{(k)}\rangle+\langle \psi_0^{\epsilon}, w^{(k)}\rangle.
\ee
Since $\phi_i^{\epsilon}\in \mR[x]$, and $\psi_0^\epsilon$,
$\psi_j ^\epsilon $ are SOS, the constraints
\[
\mathscr{V}_{c_{j}}^{(2k)}[w] = 0~  (j\in \mathcal{E}), \quad
L_{c_{j}}^{(k)}[w] \succeq 0~(j\in \mathcal{I}), \quad M_k[w] \succeq 0
\]
imply that
\[
\langle \phi_j^{\epsilon}c_j, w^{(k)}\rangle=0 ~(j \in \mathcal{E}),~
\langle \psi_j^{\epsilon} c_j, w^{(k)}\rangle\geq0 ~(j \in \mathcal{I}) , ~
\langle \psi_0^{\epsilon}, w^{(k)}\rangle \geq0.
\]
Hence, we  have that
\be
\lim\limits_{\epsilon\rightarrow 0}\langle \psi_j^{\epsilon} c_j, w^{(k)}\rangle=0, ~\lim\limits_{\epsilon\rightarrow 0}\langle \psi_0^{\epsilon}, w^{(k)}\rangle=0.
\ee
By Lemma 2.5 of \cite{nie2013certifying}, we can get for $k$ big enough,
\be \label{asym1}
\lim\limits_{\epsilon\rightarrow 0}\langle \psi_j^{\epsilon} c_j\phi_0, w^{(k)}\rangle=0, ~\lim\limits_{\epsilon\rightarrow 0}\langle \psi_0^{\epsilon}\phi_0, w^{(k)}\rangle=0.
\ee
 Combining the equations \reff{per} and \reff{asym1}, it holds that
\be
\baray{rl}
& \langle \phi_0 f_{\theta^*}, w^{(k)}\rangle  \\
= &\lim\limits_{\epsilon\rightarrow 0} \big (\sum\limits_{j \in \mathcal{E}}
   \langle \phi_j^{\epsilon}c_j\phi_0, w^{(k)}\rangle +
    \sum\limits_{j \in \mathcal{I}} \langle \psi_j^{\epsilon}
    c_j\phi_0, w^{(k)}\rangle+\langle \psi_0^{\epsilon}\phi_0, w^{(k)}\rangle \big )\\
= & 0.\\
\earay
\ee
Hence, we can get
$$
\langle\phi_0 f_{\theta^*}, w^{(k)}\rangle=0.
$$
It follows from \reff{realzero} that
\be\label{realzero1}
\langle h_t^{2 \ell}, w^{(k)}\rangle+\sum_{j \in \mathcal{E}}\langle\phi_j c_j, w^{(k)}\rangle+\sum_{j \in \mathcal{I}}\langle\psi_j c_j, w^{(k)}\rangle+\langle\psi_0, w^{(k)}\rangle=0.
\ee
Since $h_t^{2 \ell}$, $\psi_j$ ($j \in \mathcal{I}$), $\psi_0$ are SOS, we can easily verify that
\[
\langle h_t^{2 \ell}, w^{(k)}\rangle\geq 0,~ \langle\phi_j c_j, w^{(k)}\rangle=0~ (j \in \mathcal{E}),
\]
\[
\langle\psi_0, w^{(k)}\rangle \geq0,~\langle\psi_j c_j, w^{(k)}\rangle \geq 0 ~ (j \in \mathcal{I}).
\]
Hence, $\langle h_t^{2 \ell}, w^{(k)}\rangle=0$ and  $h_t \in \operatorname{ker} M_k[w^{(k)}]$ (cf. \cite{LLR08, Lau09, niebook}).

Finally, we show that $\left.w^{(k)}\right|_{2 k-2}$ is flat.
Since $S$ is finite, we know that  the quotient space $\mathbb{R}[x] / I(S)$
is finite-dimensional (see Theorem 2.6 of \cite{Lau09}).
Let $\left\{q_1, \ldots, q_T \right\}$ be a  basis of $\mathbb{R}[x] / I(S)$.
For every $\alpha \in \mathbb{N}^n$, we have that
$$
x^\alpha=\sum_{i=1}^T \beta_i q_i+\sum_{t=1}^r p_t h_t,
\quad \operatorname{deg}\left(p_t h_t\right) \leq|\alpha|,
$$
where $\beta_i\in \mR$, $p_t\in \mR[x]  $.
Since $h_t \in \operatorname{ker}  M_k[w^{(k)}]$,
if $ |\alpha| \leq k-1$, we have
\[
p_t h_t \in \operatorname{ker}  M_k[w^{(k)}], \quad
x^\alpha-\sum_{i=1}^T \beta_i q_i \in \operatorname{ker} M_k[w^{(k)}] .
\]
Let $d_q:=\max\limits_{i \in [T]} \operatorname{deg} (q_i)$.
If $d_q+1 \leq|\alpha| \leq k-1$,
then each $\alpha$th column  of $ M_k[w^{(k)}]$
is a linear combination of its first $d_q$ columns.
When $k-1-d_K \geq d_q$, we have
\[
\operatorname{rank} M_{k-1-d_K}[w^{(k)}]=\operatorname{rank} M_{k-1}[w^{(k)}].
\]
This implies that  $w^{(k)}$ has a flat truncation when $k$ is big enough.
\end{proof}

\noindent
We would like to make the following remarks:
\begin{itemize}

\item [i)]	
Theorem~\ref{finite:gpm} gives sufficient conditions for the Moment-SOS
hierarchy to have finite convergence for solving the GMP.
It is typically hard to check these optimality conditions as a priori or computationally,
since the maximizer $\theta^*$ of \reff{gpm:dual} is usually not known.
However, there is no need to check them in the computational practice
of Moment-SOS relaxations.  As in Conjecture~\ref{GMP:conj},
we conjecture that these optimality conditions are satisfied
when the defining polynomials of the GMP have generic coefficients.

\item [ii)]
In Theorem \ref{finite:gpm}, we assume that the optimal value of
\reff{gpm:dual} is achievable.
By the classical linear conic optimization theory (see \cite{Bar02}),
if the Slater condition holds for \reff{gpm},
the dual optimization \reff{gpm:dual} has a maximizer.
The maximizer of \reff{finite:gpm} may not be unique.
When \reff{gpm:dual} has more than one maximizer,
Theorem~\ref{finite:gpm} only requires
Assumption~\ref{assum} to hold for one maximizer $\theta^*$ of \reff{gpm:dual},
not necessarily all.
This is because the proof of Theorem \ref{finite:gpm}
only uses optimality conditions of \reff{stand:opt} for $\theta^*$. We refer to Example \ref{ex3.4} for such an example.

\item [iii)]
In Theorem \ref{finite:gpm},  we assume that the dual optimization \reff{gpm:dual}
has an interior point, i.e., there exists a feasible point $\bar{\theta}\in \mR^m$
such that $f_{\bar{\theta}}>0$ on $K$.
Then, the strong duality holds for \reff{gpm}--\reff{gpm:dual}.	
There also exist other type conditions,
such as closedness of the image of the primal cone
under the primal affine map (see \cite[Theorem 7.2]{Bar02}),
which can also guarantee the strong duality.
In our proof, the existence of
an interior point for \reff{gpm:dual} is quite important,
more than just guaranteeing the strong duality.
For instance, it is used to get the inclusion relation \reff{cwtc}.

\end{itemize}

\subsection{The special case of polynomial optimization}
\label{sp:po}

Consider the polynomial optimization problem
\be  \label{po}
\left\{ \baray{rl}
\min & f(x)  \\
\st &  	c_{j}(x)=0~(j \in \mathcal{E}), \\
&c_{j}(x) \geq 0~(j \in \mathcal{I}),
\earay \right.
\ee
for given polynomials $f, c_j$. Let $f_{min}$ denote the optimal value of \reff{po}.
It is interesting to note that \reff{po}
is equivalent to the moment optimization
\be  \label{po:mom}
\left\{ \baray{rl}
\min & \langle f, y \rangle   \\
\st &  \langle 1, y \rangle = 1,  \\
&  y \in \mathscr{R}_d(K).
\earay \right.
\ee
In the above, the equality $\langle 1, y \rangle = 1$
means that every representing measure $\mu$ for $y$
must have unit mass (i.e., $\mu(K) = 1$).
The dual optimization of \reff{po:mom} is
\be  \label{po:sos}
\left\{ \baray{rl}
\max & \theta   \\
\st & f - \theta \cdot 1 \in \mathscr{P}_d(K), \\
    &  \theta \in \re.
\earay \right.
\ee
The optimal value $\theta^*$ of \reff{po:sos}
is the minimum value $f_{min}$ of \reff{po}.
The $k$th order moment relaxation for \reff{po:mom} is
\be  \label{pop:momrelax}
\left\{ \baray{cl}
\min &  \langle f, w \rangle  \\
\st  &  \langle 1,w\rangle  = 1,    \\
& \mathscr{V}_{c_{j}}^{(2k)}[w] = 0~ (j\in \mathcal{E}), \\
&L_{c_{j}}^{(k)}[w] \succeq 0~(j\in \mathcal{I}), \\
& M_k[w] \succeq 0, \, w \in \mathbb{R}^{\mathbb{N}_{2 k}^{n}} .
\earay \right.
\ee
When Theorem \ref{finite:gpm} is applied to \reff{po:mom}-\reff{po:sos},
we can get the following theorem.

\begin{thm} \label{poly:fin}
Suppose $\ideal{c_{eq}}+\qmod{c_{in}}$ is  archimedean.
If the LICQC, SCC, SOSC hold at every  minimizer of \reff{po},
then  every minimizer of the moment relaxation \reff{pop:momrelax}
must have a flat truncation when $k$ is sufficiently large.
\end{thm}

When $\ideal{c_{eq}}$ is a real radical ideal,
the above result appeared in \cite[Section~5.4]{niebook}.
However, the real radicalness of
$\ideal{c_{eq}}$ is not required in Theorem~\ref{poly:fin}.
This improves the results in the earlier work
\cite{nie2013certifying,niebook}.
When there exists a equality constraint $c_j(x)=0$
for the set $K$, we can equivalently replace it by two inequalities
$c_j(x)\geq 0$, $-c_j(x)\geq0$. However, this reformulation will
make the linear independence constraint qualification
fail at each minimizer of \reff{po}, and we can not get
the finite convergence by using this formulation.

\bigskip
\noindent
{\bf Remark:}
In Theorem~\ref{finite:gpm}, we assume that the LICQC,  SCC and SOSC
hold at every minimizer of \reff{stand:opt},
which depends on the  maximizer $\theta^*$ of \reff{gpm:dual}.
When the GMP specializes to the polynomial optimization problem \reff{po},
the optimization \reff{stand:opt} reads
\be  \label{stand:opt1}
\left\{ \baray{rl}
\min & f(x)-f_{\min}   \\
\st &  	c_{j}(x)=0~(j \in \mathcal{E}), \\
& c_{j}(x) \geq 0~(j \in \mathcal{I}).
\earay \right.
\ee
Clearly, \reff{stand:opt1} is equivalent to \reff{po},
since their objectives differ only by the constant $f_{\min}$.
Thus, the LICQ, SCC, SOSC hold at every minimizer of \reff{stand:opt1}
if and only if they hold at every minimizer of \reff{po},
which is independent of the maximizers of \reff{gpm:dual}.

\subsection{Some exposition examples}

As shown in Theorem~\ref{finite:gpm}, if the LICQC, SCC, SOSC hold
at every minimizer of \reff{stand:opt},
not only the Moment-SOS hierarchy has finite convergence,
but also the flat truncation is satisfied.
This is observed in the following example.

\begin{example}\label{ex3.4}
For polynomials
\[
a_1=x_1^2 x_2^4+x_2^2 x_3^4+x_3^2 x_1^4,~ a_2=x_1^3 x_2^3+x_2^3 x_3^3+x_3^3 x_1^3,
\]
\[
a_3=x_1^5 x_2+x_2^5 x_3+x_3^5 x_1,
\]
we consider the GMP
\be  \label{ex1}
\left\{ \baray{rl}
\min & \langle x_1^6+x_2^6+x_3^6, y\rangle  \\
\st  &\langle a_i,y\rangle = 1,~i=1,2,3,\\
&   y \in \mathscr{R}_{6}(K),
\earay \right.
\ee
where $K$ is the unit sphere
$\{x\in \mR^3\mid x_1^2+x_2^2+x_3^2-1=0\}.$
The optimal value $\phi^*=1$. The dual problem of \reff{ex1} is
\be  \label{ex2:dual}
\left\{ \baray{rl}
\max & \theta_1+\theta_2+\theta_3  \\
\st &   x_1^6+x_2^6+x_3^6- \theta_1a_1-\theta_2a_2-\theta_3a_3
\in \mathscr{P}_{6}(K).
\earay \right.
\ee
For $\theta=(0,0,0)$, we have $f_{\theta}(x)=x_1^6+x_2^6+x_3^6\in int(\mathscr{P}_{6}(K))$,
i.e., the Slater condition holds. We can verify that  $\{ \theta_1+\theta_2+\theta_3=1,~\theta_i\geq 0~ (i=1,2,3)\}$ are all maximizers of \reff{ex2:dual}, which is infinite. For one maximizer  $\theta^*=(0,1,0)$, the maximum problem \reff{stand:opt} reads
\be  \label{ex2:opt}
\left\{ \baray{rl}
\min & x_1^6+x_2^6+x_3^6- (x_1^3 x_2^3+x_2^3 x_3^3+x_3^3 x_1^3) \\
\st
&x_1^2+x_2^2+x_3^2-1=0.
\earay \right.
\ee
All minimizers of \reff{ex2:opt} are
\[
(\frac{1}{\sqrt{3}},\frac{1}{\sqrt{3}},\frac{1}{\sqrt{3}}),
(\frac{-1}{\sqrt{3}},\frac{-1}{\sqrt{3}},\frac{-1}{\sqrt{3}}).
\]
 	The unit sphere is smooth, so the constraint qualification condition holds at every
 	feasible point. There is no inequality constraint, so strict complementarity is automatically satisfied. We can verify that the second order sufficient condition
 holds at both minimizers.
For instance, at the minimizer $u=(\frac{1}{\sqrt{3}},\frac{1}{\sqrt{3}},\frac{1}{\sqrt{3}})$,
$$
\nabla_x^2 L(u)=3 \cdot\left[\begin{array}{lll}
1 & 0 & 0 \\
0 & 1 & 0 \\
0 & 0 & 1
\end{array}\right]-\left[\begin{array}{l}
1 \\
1 \\
1
\end{array}\right]\left[\begin{array}{l}
1 \\
1 \\
1
\end{array}\right]^T, \quad G(u)^{\perp}=\left[\begin{array}{l}
1 \\
1 \\
1
\end{array}\right]^{\perp}.
$$
Clearly, the SOSC is satisfied at $u$.
By Theorem \ref{finite:gpm},
the hierarchy of \reff{momre:gpm}-\reff{sos:gpm} has finite convergence.
A numerical experiment by GloptiPoly~3 \cite{2009GloptiPoly}
indicates that $\vartheta_3= \phi_3=1$.
\end{example}

We would like to remark that if the ideal $\ideal{c_{eq}}$ is real radical,
the finite convergence of the Moment-SOS hierarchy
can be implied by \cite{nieopcd}. However, when $\ideal{c_{eq}}$ is not real radical,
the finite convergence still holds.
This is shown in Theorem \ref{finite:gpm}.
The following is such an example.

\begin{example}
We consider the GMP
\be  \label{ex3}
\left\{ \baray{rl}
\min & \langle x_1+\dots+x_6, y\rangle  \\
\st &  \langle x_1^3+\cdots+x_6^3,y\rangle \geq1, \\
    & \langle x_1^4+\cdots+x_6^4,y\rangle =1, \\
&   y \in \mathscr{R}_{4}(K),
\earay \right.
\ee
where the set $K$ is
\be \nonumber
\baray{ll}
K \quad = \quad \big \{x\in \mR^6:& x _1^2+\dots+x_6^2-2=0,~  2(x_1^2+x_2^2+x_3^2)-x_4=0, \\
&x_4-2(x_5^2+x_6^2)=0,~ x_1 \ge 0, \ldots, x_6 \ge 0 \big \}.
\earay
\ee
The optimal value $\phi^*=\frac{2+2\sqrt{2}}{3}$. One can see that $x_4-1\in \ideal{V_{\mR}(c_{eq})}$.
However,  $x_4-1 \notin \ideal{c_{eq}}$. Suppose otherwise it were true,
then there exist $h_1,h_2,h_3\in \mR[x]$ such that
\be \nonumber
\baray {ll}
x_4-1=&h_1\cdot(x_1^2+\dots+x_6^2-2)+h_2\cdot(2(x_1^2+x_2^2+x_3^2)-x_4)\\
&+h_3\cdot(x_4-2(x_5^2+x_6^2)).
\earay
\ee
Substitute $x$ for $(\frac{\sqrt{3}\mathrm{\bf i}}{3},\frac{\sqrt{3}\mathrm{\bf i}}{3},\frac{\sqrt{3}\mathrm{\bf i}}{3},-2,\frac{\sqrt{2}\mathrm{\bf i}}{2},\frac{\sqrt{2}\mathrm{\bf i}}{2})$ ($\mathrm{\bf i}$ is the imaginary unit) in the above identity,
we get the contradiction $-3=0$.
Thus, the ideal $\ideal{c_{eq}}$ is not real radical. The dual problem of \reff{ex3} is

\be  \label{ex3:dual}
\left\{ \baray{rl}
\max & \theta_1+\theta_2  \\
\st &    \sum\limits_{i=1}^6x_i-  \theta_1 \sum\limits_{i=1}^6x_i^3-\theta_2\sum\limits_{i=1}^6x_i^4
\in \mathscr{P}_{4}(K),\\
&   \theta\in \mR^2,~\theta_1 \geq 0.
\earay \right.
\ee
  For $\theta=(0,0)$, we have $f_{\theta}(x)=\sum\limits_{i=1}^6x_i\in int(\mathscr{P}_4(K))$,
so the Slater condition holds.
One can verify that the maximizer of \reff{ex3:dual} is
$\theta^*=(0,\frac{2+2\sqrt{2}}{3})$.
The corresponding optimization problem \reff{stand:opt} reads
\be  \label{ex3:dualmax}
\left\{ \baray{rl}
\min &\sum\limits_{i=1}^6x_i-\frac{2+2\sqrt{2}}{3}\sum\limits_{i=1}^6x_i^4 \\
\st &   x_1^2+\dots+x_6^2-2=0,\\
&2(x_1^2+x_2^2+x_3^2)-x_4=0,\\
&x_4-2(x_5^2+x_6^2)=0,\\
& x_1\geq0,\dots,~x_6\geq 0.
\earay \right.
\ee
The set of minimizers for \reff{ex3:dualmax} consists of the points:
\[
(\frac{\sqrt{2}}{2},0,0,1,\frac{\sqrt{2}}{2},0),~(0,\frac{\sqrt{2}}{2},0,1,
\frac{\sqrt{2}}{2},0),~(0,0,\frac{\sqrt{2}}{2},1,\frac{\sqrt{2}}{2},0),
\]
\[
(\frac{\sqrt{2}}{2},0,0,1,0,\frac{\sqrt{2}}{2}),~(0,\frac{\sqrt{2}}{2},0,1,0,
\frac{\sqrt{2}}{2}),~(0,0,\frac{\sqrt{2}}{2},1,0,\frac{\sqrt{2}}{2}).
\]
The LICQC, SCC, SOSC hold at all these minimizers.
A numerical experiment by GloptiPoly~3 indicates that $\vartheta_3= \phi_3=\phi^*$.
\end{example}

We remark that in Theorems~\ref{finite:gpm} and \ref{poly:fin},
none of the optimality conditions there can be dropped.
We refer to \cite{nieopcd} for such counterexamples.

\section{The Moment-SOS hierarchy for  GMPs with unbounded sets}
\label{sc:homo}

When the feasible set $K$ is unbounded, the Moment-SOS hierarchy \reff{momre:gpm}--\reff{sos:gpm}
may not converge. To fix this issue, a new type Moment-SOS hierarchy based on homogenization
is proposed in \cite{hny}. Its finite convergence is shown
if the ideal of equality constraining polynomials is real radical,
provided the LICQC, SCC, SOSC hold at every minimizer.
In this section, we prove the same conclusion
without the real radicalness assumption.
In particular, this resolves Conjecture~8.2 raised in \cite{hny} when it specializes to polynomial optimization.
We also prove similar results for the finite convergence of
the Moment-SOS hierarchy with denominators,
induced by Putinar-Vasilescu's Positivstellensatz
\cite{MLM21,maim,putinar1999positive}.

We begin with Moment-SOS relaxations for solving
generalized moment problems with unbounded sets.

\subsection{Homogenization of GMPs}
\label{ssc:unbgmp}

This subsection generalizes the homogenization trick given in \cite{hny,hny1}
to solve the GMP with unbounded sets
and proves its finite convergence
without the real radicalness assumption.

For a polynomial $p(x)$ of degree $\ell$, let $\tilde{p}$
denote its homogenization in  $\tilde{x} \coloneqq \left(x_0,x\right)$,
i.e., $\tilde{p}(\tilde{x}) = x_0^\ell p(x/x_0)$.
For the set $K$ as in \reff{feas:set},
its homogenization $\widetilde{K}$ is the set
\be \label{set:wtld:K}
\widetilde{K} \coloneqq \left\{\tilde{x} \in \mathbb{R}^{n+1}
\left| \baray{l}
\tilde{c}_{j}(\tilde{x})=0~(j \in \mathcal{E}), \\
\tilde{c}_{j}(\tilde{x}) \geq 0~(j \in \mathcal{I}), \\
\|\tilde{x}\|^2-1=0, ~ x_{0} \geq 0
\earay  \right.
\right\} .
\ee
Note that $\widetilde{K}$ is always compact.
The set $K$ is said to be {\em closed at infinity}  if
\be \label{closed:inf} 
 \widetilde{K} \,= \,  cl(\widetilde{K}\cap \{x_0>0\}).
\ee

When $K$ is  closed at infinity, a polynomial $f\geq 0$ on $K$ if and only if
its homogenization $\tilde{f} \geq 0$ on $\widetilde{K}$;  see \cite{hny1,njwdis}.
Denote the polynomial tuples
\be \label{sets:tld:ceqcin}
\tilde{c}_{eq}  \coloneqq  \big\{ \tilde{c_j}(\tilde{x}) \big\}_{j\in \mc{E} }
\cup \{ \|\tilde{x}\|^2-1 \}, \quad
\tilde{c}_{in}  \coloneqq  \big\{ \tilde{c_j}(\tilde{x}) \big\}_{j \in \mc{I} }
\cup \{ x_0 \}.
\ee
Recall that $d$ is the degree as in \reff{gpm}-\reff{gpm:dual}.
Denote the degree-d homogenization
\be \label{hat:fai}
\hat{f}  \, \coloneqq \,  x_0^{d}f( x/x_0 ), \quad
\hat{a}_i  \, \coloneqq \,  x_0^{d}a_i( x/x_0 )~(i=1\dots,m).
\ee
The homogenization of the generalized moment problem \reff{gpm} is
\be  \label{hom:gpm}
\left\{ \baray{rl}
\min & \langle \hat{f}, z\rangle  \\
\st & \langle \hat{a}_i,z\rangle   =  b_i \, \, ( 1 \le i \le m_1),\\
&\langle \hat{a}_i,z\rangle   \geq  b_i \, \, ( m_1 < i \le m), \\
&  z \in \mathscr{R}_{d}(\widetilde{K}).
\earay \right.
\ee
Similarly,  the homogenized optimization of \reff{gpm:dual} is
\be  \label{hom:psd}
\left\{ \baray{rl}
\max & \sum\limits_{i=1}^m b_i\theta_i  \\
\st & \hat{f} - \sum\limits_{i=1}^m \theta_i\hat{a}_i
\in \mathscr{P}_{d}(\widetilde{K}),\\
& \theta\in \mR^m,~\theta_i \geq 0 \,\, (m_1 < i \le m).
\earay \right.
\ee
Let  $\hat{\phi}^*$, $\hat{\vartheta}^*$ denote the optimal values of
 \reff{hom:gpm}, \reff{hom:psd} respectively.
By the weak duality, $\hat{\vartheta}^* \le \hat{\phi}^*$.
Denote
\[
\hat{f}_{\theta} \, \coloneqq \,
\hat{f} - \sum\limits_{i=1}^m \theta_i\hat{a}_i,
\]
\[
H = \Big \{ f - \sum\limits_{i=1}^m \theta_ia_i:
\theta\in \mR^m,~\theta_i \ge 0 \, \mbox{~for~} \,
 m_1 < i \le m  \Big \}.
\]
The following is the relationship between  the optimal values of the optimization \reff{gpm}--\reff{gpm:dual} and the homogenized   optimization  \reff{hom:gpm}--\reff{hom:psd}.

\begin{prop}
\label{pro:eq}
Suppose $K$ is closed at infinity. Then, we have
\[
\hat{\vartheta}^* = \vartheta^*, \quad  \hat{\phi}^*  \leq  \phi^* .
\]
Furthermore, if $H \cap int(\mathscr{P}_{d}(K)) \ne \emptyset$, then
\[
\hat{\vartheta}^*= \vartheta^*=\hat{\phi}^*=\phi^*.
\]
\end{prop}
\begin{proof}
Since $K$ is closed at infinity, the polynomial
$f - \sum\limits_{i=1}^m \theta_i a_i \geq 0$ on $K$
if and only if $ \hat{f}_{\theta} \geq 0$ on $\widetilde{K}$.
So, $\hat{\vartheta}^* = \vartheta^*$.

Suppose the tms $y\in \mR^{\mathbb{N}_{d}^{n}}$ is a feasible point of \reff{gpm}.
Then there exist  $u_t\in K$,  $\lambda_t \ge 0$ $(t=1,\dots,r)$ such that
(see \cite[Theorem~8.1.1]{niebook})
\be \label{decom:w}
y = \lambda_{1} [u_1]_{d} + \cdots +
\lambda_{r}[u_r]_{d} .
\ee
For each $t=1,\dots,r$, let
\[
\tilde{\lambda}_t = \lmd_t \|(1,u_t)\|^{d}, ~(\tau_t, v_t)= (1,u_t)/\sqrt{1+\|u_t\|^2}.
\]
Then we can see that
\[
\tilde{c}_j(\tau_t, v_t)=\tau_t^{\deg(c_j)}c_j(\frac{v_t}{\tau_t})=\tau_t^{\deg(c_j)}c_j(u_t)= 0~(j\in \mathcal{E}),
\]	
\[
\tilde{c}_j(\tau_t, v_t)=\tau_t^{\deg(c_j)}c_j(\frac{v_t}{\tau_t})=\tau_t^{\deg(c_j)}c_j(u_t)\geq  0~(j\in \mathcal{I}).
\]
Hence, we know $(\tau_t, v_t)\in \widetilde{K}$. Let
\be \label{atom:mu1}
z =\tilde{\lambda}_{1}[(\tau_1, v_1)]_{d}  + \cdots +
\tilde{\lambda}_{r} [(\tau_r, v_r)]_{d}.
\ee
A direct computation shows that
\be \nonumber
\baray{llll}
\langle \hat{f}, z\rangle & = \sum\limits_{t=1}^r \tilde{\lambda}_{t}\langle \hat{f}, [(\tau_t, v_t)]_{d}\rangle &= \sum\limits_{t=1}^r \tilde{\lambda}_{t} \hat{f}(\tau_t, v_t) & \\
&=\sum\limits_{t=1}^r \tilde{\lambda}_{t}\tau_t^d f(u_t)&
 =\sum\limits_{t=1}^r \lambda_{t}f(u_t) &= \langle f, y\rangle .
\earay
\ee
Similarly, we also have
\[
\langle \hat{a}_i,z\rangle   = \langle a_i,y\rangle= b_i \, \, ( 1 \le i \le m_1),
\]
\[
\langle \hat{a}_i,z\rangle = \langle a_i,y\rangle  \geq  b_i \, \, ( m_1 < i \le m).
\]
This shows that $z$ is a feasible point of \reff{hom:gpm}
and $\langle \hat{f}, z\rangle=\langle f, y\rangle$.
The above holds for every feasible point $y$ of \reff{gpm},
so $\hat{\phi}^*  \le  \phi^*.$

When $H \cap int(\mathscr{P}_{d}(K)) \ne \emptyset$,
%
%
the strong duality holds between \reff{gpm:dual} and \reff{gpm},
i.e., $\phi^* = \vartheta^*$, so,
\[
  \vartheta^* =   \phi^* \ge  \hat{\phi}^*  \ge  \hat{\vartheta}^* =   \vartheta^*.
\]
Therefore, all these values must be equal to each other.
\end{proof}

\begin{remark}
In the proof of Proposition~\ref{pro:eq}, we have seen that
if $y$ as in \reff{decom:w} is feasible for \reff{gpm},
then $z$ as in \reff{atom:mu1} is feasible for \reff{hom:gpm}.
The reverse of this is also true under some conditions. Suppose
\be \nonumber
z  =\tilde{\lambda}_{1}[(\tau_1, v_1)]_{d}  + \cdots +
\tilde{\lambda}_{r} [(\tau_r, v_r)]_{d},
\ee
is feasible for \reff{hom:gpm}, with $\tilde{\lmd}_t \ge 0$,
$(\tau_t, v_t) \in \widetilde{K}$ and all $\tau_t > 0$. Let
\[
\lmd_t =  \tilde{\lambda}_t \tau_t^{d}, ~
u_t= v_t / \tau_t.
\]
Then $y = \lambda_{1} [u_1]_{d} + \cdots + \lambda_{r}[u_r]_{d}$
is a feasible point of \reff{gpm}
and $\langle \tilde{f},z\rangle = \langle f, y \rangle$.
Therefore, if the above $z$ is a minimizer of \reff{hom:gpm},
then $y$ is also a minimizer of \reff{gpm}.
\end{remark}

For a degree $k \geq d_0$, the $k$th  order  moment relaxation of \reff{hom:gpm} is
\be \label{hom:mom:rel}
\left\{ \baray{cl}
\min &  \langle  \hat{f}, w \rangle  \\
\st  &  \langle  \hat{a}_i,w\rangle  =  b_i \, \, ( 1 \le  i \le m_1),  \\
&\langle  \hat{a}_i,w\rangle   \geq  b_i \, \, (m_1 < i \le m),\\
&\mathscr{V}_{1-\|\tilde{x}\|^2}^{(2k)}[w] = 0,~\mathscr{V}_{\tilde{c}_{j}}^{(2k)}[w] = 0~ (j\in \mathcal{E}), \\
&L_{x_0}^{(k)}[w] \succeq 0,~L_{\tilde{c}_{j}}^{(k)}[w] \succeq 0~(j\in \mathcal{I}), \\
& M_k[w] \succeq 0, \, w \in \mathbb{R}^{\mathbb{N}_{2 k}^{n+1}} .
\earay \right.
\ee
The $k$th order SOS relaxation of \reff{hom:psd} is
\be \label{hom:dual:rel}
\left\{ \baray{cl}
\max &  b_1 \theta_1 + \cdots + b_m \theta_m  \\
\st &  \hat{f} - \sum\limits_{i=1}^m \theta_i \hat{a}_i
\in  \ideal{\tilde{c}_{eq}}_{2k}+\qmod{\tilde{c}_{in}}_{2k},\\
& \theta \in \mR^m,~ \theta_{m_1+1} \geq 0,\ldots, \theta_{m} \geq 0 .
\earay \right.
\end{equation}
Denote the optimal values of  \reff{hom:mom:rel}, \reff{hom:dual:rel}
by  $\hat{\phi}_k$, $\hat{\vartheta}_k$ respectively.

\begin{thm}
 \label{hom:fin}
Suppose $K$ is closed at infinity,
$H \cap int(\mathscr{P}_{d}(K)) \ne \emptyset$ and the minimizer $y^*$
of \reff{gpm} satisfies $y^*\neq 0$.
Assume $\theta^* \in \mR^m$ is a maximizer of \reff{gpm:dual}.
If the LICQC, SCC, SOSC hold at every minimizer of the optimization
$(\tilde{x} \coloneqq (x_0, x))$
\be  \label{hstand:opt}
\left\{ \baray{rl}
\min & \hat{f}(\tilde{x}) - \sum\limits_{i=1}^m \theta_i^* \hat{a}_i (\tilde{x}) \\
\st &  	\tilde{c}_{j}(\tilde{x})=0~(j \in \mathcal{E}), \\
&\tilde{c}_{j}(\tilde{x}) \geq 0~(j \in \mathcal{I}).
\earay \right.
\ee
Then,  we have:
\bit

\item [(i)] The Moment-SOS hierarchy \reff{hom:mom:rel}--\reff{hom:dual:rel}
has finite convergence, i.e., for all $k$ sufficiently large,
\[
\hat{\vartheta}_k = \hat{\phi}_k = \vartheta^* = \phi^* .
\]

\item [(ii)] Every minimizer of the moment relaxation \reff{hom:mom:rel}
must have a flat truncation, when $k$ is sufficiently large.

\eit
\end{thm}

\begin{proof}
It follows from Proposition~\ref{pro:eq} that  $\theta^* \in \mR^m$
is also a maximizer of \reff{hom:psd} and
$\hat{\vartheta}^*= \vartheta^*=\hat{\phi}^*=\phi^*.$
By Theorem 3.4 \cite{hny1}, the condition $H \cap int(\mathscr{P}_{d}(K)) \ne \emptyset$
implies that there exists $\bar{\theta}\in \mR^m$ with
$\theta_{m_1+1}\geq 0,\dots,\bar{\theta}_m\geq 0$
such that $\hat{f}_{\bar{\theta}}>0$ on $\widetilde{K}$,
i.e., the Slater condition holds for optimization \reff{hom:psd}.
Thus, Assumption \ref{assum} holds for  optimization \reff{hom:psd}.
Note that $\ideal{\tilde{c}_{eq}}+\qmod{\tilde{c}_{in}}$ is archimedean.
The conclusions follow from Theorem \ref{finite:gpm}.
\end{proof}

\bigskip \noindent
{\bf Remark:}
To verify that $K$ is closed at infinity,  we need to check the condition~\reff{closed:inf},
which can be quite tricky and depends on the choice of
describing polynomials for $K$ (see \cite{njwdis}).
However, being closed at infinity is a generic property for semialgebraic sets
(see \cite{guo2014minimizing}).
%
Most frequently appearing unbounded semialgebraic sets are closed at infinity.
For instance, the nonnegative orthant $\re_+^n$,
the exterior ball $\{ x^Tx \geq 1\}$, and
the exterior box $\{ x_1^2\geq 1,\dots,x_n^2\geq 1\}$
are all closed at infinity.

\bigskip
The following is an illustrated example for the GMP with unbounded sets,
where the ideal $\ideal{\tilde{c}_{eq}}$ is not real radical.
\begin{example}
We consider the GMP
\be  \label{ex4:1}
\left\{ \baray{rl}
\min & \langle x_1^4+(2x_2-1)^4+(2x_3-1)^4, y\rangle  \\
\st  &\langle x_2^2+x_3^2, y\rangle = 1,\\
 &\langle x_1^2x_2+x_2^2x_3+x_3^2x_1, y \rangle \geq \frac{1}{5},\\
&   y \in \mathscr{R}_{4}(K),
\earay \right.
\ee
with the unbounded set from \cite{bhl}:
\[
K=\{x\in \mR^3:x_1x_2x_3=0,~x_3(x_1^2+x_2^2+x_3^2+x_2)=0,~x_2(x_2+x_3)=0\}.
\]
The optimal values $\phi^*=\hat{\phi}^*=32$. One can check that
$x_1x_3\in \ideal{V_{\mR}(\tilde{c}_{eq})}$ but
$x_1x_3\notin \ideal{\tilde{c}_{eq}}$. Suppose otherwise it were true,
then there exist $h_1,h_2,h_3,h_4\in \mR[\tilde{x}]$ such that
\be \nonumber
\baray{ll}
x_1x_3&=h_1 \cdot x_1x_2x_3 +h_2\cdot x_3(x_1^2+x_2^2+x_3^2+x_2x_0)\\
&+h_3\cdot x_2(x_2+x_3)+h_4\cdot (x_0^2+x_1^2+x_2^2+x_3^2-1) .
\earay
\ee
Substituting $(x_0,x_2)$ for $(1,0)$ in the above,
we get the representation
\[
x_1x_3= \Big(h_2(1,x_1,0,x_3)x_3+h_4(1,x_1,0,x_3) \Big)(x_1^2+x_3^2),
\]
which can never hold.
So, $\ideal{\tilde{c}_{eq}}$ is not real radical.
The dual optimization problem of \reff{ex4:1} is
\be  \label{ex2:dual4}
\left\{ \baray{rl}
\max & \theta_1+\frac{1}{5}\theta_2 \\
\st &   x_1^4+(2x_2-1)^4+(2x_3-1)^4- \theta_1(x_2^2+x_3^2) \\
    &  \qquad \qquad \qquad     -\theta_2(x_1^2x_2+x_2^2x_3+x_3^2x_1)
          \in \mathscr{P}_{4}(K),\\
&\theta\in \mR^2,~ \theta_2\geq0.
\earay \right.
\ee
For $\theta=(0,0)$, we have $f_{\theta}=x_1^4+(2x_2-1)^4+(2x_3-1)^4 \in int(\mathscr{P}_{4}(K))$,
so $H \cap int(\mathscr{P}_{4}(K)) \ne \emptyset$.
A maximizer of \reff{ex2:dual4} is $\theta^*=(32,0)$ and
the optimization problem~\reff{hstand:opt} reads
\be  \nn \label{ex2:opt4}
\left\{ \baray{rl}
\min &x_1^4+(2x_2-x_0)^4+(2x_3-x_0)^4 -32(x_2^2+x_3^2)x_0^2 \\
\st
&x_1x_2x_3=0,~x_2(x_2+x_3)=0,\\
&x_3(x_1^2+x_2^2+x_3^2+x_2x_0)=0,\\
&x_0^2+x_1^2+x_2^2+x_3^2=1,\\
& x_0\geq0.
\earay \right.
\ee
Its unique minimizer is $(\frac{\sqrt{6}}{3},0,\frac{-\sqrt{6}}{6},\frac{\sqrt{6}}{6})$,
and we can verify that the LICQC, SCC, SOSC all hold at it.
By Theorem \ref{hom:fin},
the hierarchy \reff{hom:dual:rel}-\reff{hom:mom:rel} has finite convergence.
A numerical experiment by GloptiPoly~3 indicates that $\hat{\vartheta_2}= \hat{\phi}_2=32$.
\end{example}

\subsection{Polynomial optimization with unbounded sets}
\label{ssc:pop}

Consider the polynomial optimization problem
\be  \label{po4}
\left\{ \baray{rl}
\min & f(x)  \\
\st &  	c_{j}(x)=0~(j \in \mathcal{E}), \\
&c_{j}(x) \geq 0~(j \in \mathcal{I}),
\earay \right.
\ee
for given polynomials $f,  c_j\in \mR[x]$.
We consider the case that the feasible set $K$ is unbounded.
Let $f_{min}$ denote the optimal value of \reff{po4}.
Suppose the degree of $f$ is $d$. A homogenized Moment-SOS hierarchy
is proposed in \cite{hny} for solving \reff{po4} when $K$ is unbounded.

Recall the notations $\tilde{c}_{eq}, \tilde{c}_{in}$
as in \reff{sets:tld:ceqcin}.
For a degree $k \geq \lceil \frac{d}{2}\rceil$, the $k$th order homogenized SOS relaxation
of \reff{po4} is
\begin{equation}  \label{h3.3}
	\left\{ \baray{rl}
	\max &   \gamma \\
	\st &  \tilde{f}(\tilde{x})- \gamma x_0^d \in
	\ideal{\tilde{c}_{eq}}_{2k} +\qmod{\tilde{c}_{in}}_{2k}.
	\earay \right.
\end{equation}
Its dual optimization is the $k$th order moment relaxation
\begin{equation}      \label{hd3.3}
	\left\{ \baray{rl}
	\min  &  \langle \tilde{f}, y \rangle  \\
	\st &  \langle x_0^d, w \rangle = 1, \\
	 & \mathscr{V}_{1-\|\tilde{x}\|^2}^{(2k)}[w] = 0,~\mathscr{V}_{\tilde{c}_{j}}^{(2k)}[w] = 0~ (j\in \mathcal{E}), \\
	&L_{x_0}^{(k)}[w] \succeq 0,~L_{\tilde{c}_{j}}^{(k)}[w] \succeq 0~(j\in \mathcal{I}), \\
	&  M_k[w] \succeq 0, \,
	w \in \re^{ \N^{n+1}_{2k} }.
	\earay \right.
\end{equation}
Let $f_k$ and $f^{\prime}_k$ denote the optimal values of
\reff{h3.3}, \reff{hd3.3} respectively. We remark that
\reff{h3.3}--\reff{hd3.3} is a special case of
\reff{hom:mom:rel}--\reff{hom:dual:rel} for \reff{po4}.

We consider the homogenized optimization problem
(note $\tilde{x}: = (x_0, x)$)
\be   \label{h3.5}
\left\{ \baray{rl}
\min  &  F(\tilde{x}) \coloneqq \tilde{f}(\tilde{x}) - f_{min} \cdot x_0^d  \\
\st &   \tilde{c}_{j}(\tilde{x})=0~(j \in \mathcal{E}), \\
&   x_0^2 + \| x \|^2-1=0, \\
&  \tilde{c}_{j}(\tilde{x}) \geq 0~(j \in \mathcal{I}),\\
&   x_{0} \geq 0 .
\earay \right.
\ee
Assume the optimal value $f_{min} > -\infty$ and $K$ is closed at infinity.
Note that the optimal value of \reff{h3.5} is 0, and a feasible point $x^* \in K$
is a  minimizer of $(\ref{po4})$ if and only if the point
$\tilde{x}^* \coloneqq (1+\|x^*\|^2)^{-\half} (1, x^*)$  is a minimizer of \reff{h3.5}.

In the work \cite{hny}, the Moment-SOS hierarchy \reff{h3.3}-\reff{hd3.3}
is shown to have finite convergence if $\ideal{\tilde{c}_{eq}}$
is real radical, provided the LICQC, SCC and SOSC hold at every  minimizer of \reff{h3.5}.
We remark that these optimality conditions are essentially equivalent
for \reff{po4} and \reff{h3.5}.
Here, we prove the same conclusion holds even if
$\ideal{\tilde{c}_{eq}}$ is not real radical.
This resolves Conjecture~8.2 of the work \cite{hny}.

\begin{thm} \label{finite:poin}
Suppose  $K$ is closed at infinity and $f\in \mR+ int(\mathscr{P}_{d}(K)).$
If the LICQC, SCC and SOSC hold at every minimizer of \reff{h3.5},
then the Moment-SOS hierarchy \reff{h3.3}-\reff{hd3.3} has finite convergence,
i.e., $f_k = f^{\prime}_k = f_{min}$ for all $k$ sufficiently large.
Furthermore, every minimizer of the moment relaxation \reff{hd3.3}
must have a flat truncation, when $k$ is sufficiently large.
\end{thm}
\begin{proof}
The optimization \reff{po4} is equivalent to the following GMP
\be  \nn \label{gpmpo}
\left\{ \baray{rl}
\min & \langle f, y\rangle  \\
\st & \langle 1,y\rangle   = 1,\\
&  w \in \mathscr{R}_{d}(K).
\earay \right.
\ee
%
%
The assumption $f \in \mR+ int(\mathscr{P}_{d}(K))$
implies the minimum value $f_{min} > -\infty$.
For the corresponding optimization problem \reff{gpm:dual},
the optimal value is $f_{min}$ and it is achievable.
The conclusion then follows directly from Theorem~\ref{hom:fin}.
\end{proof}

The following is an  example for the polynomial optimization with unbounded sets, where  the ideal $\ideal{\tilde{c}_{eq}}$ is not real radical.
\begin{example}
We consider the polynomial optimization
\be  \label{poex1}
\left\{ \baray{rl}
\min & x_1^6+x_2^6+3 x_1^2 x_2^2-x_1^4\left(x_2^2+1\right)-x_2^4
        \left(x_1^2+1\right)-\left(x_1^2+x_2^2\right) \\
\st  &x_1^3+x_1x_2^4+x_1=0,~x_2\geq 0.
\earay \right.
\ee
The optimal value $f_{min}=-1$. One can check that $x_1\in \ideal{V_\mR(\tilde{c}_{eq})}$ but
 $x_1 \notin \ideal{\tilde{c}_{eq}}$. Suppose otherwise it were true,
then there exist $h_1,h_2\in \mR[\tilde{x}]$ such that
\be \nonumber
x_1=h_1\cdot(x_1^3x_0^2+x_1x_2^4+x_1x_0^4)+h_2\cdot(x_1^2+x_2^2+x_0^2-1).
\ee
Substituting $x$ for $(0,1,
0)$
 in the above,
we get the contradiction $1=0$. So, the ideal $\ideal{\tilde{c}_{eq}}$ is not real radical.
The optimization problem \reff{h3.5} reads
\be  \nn \label{poex1:dual}
\left\{ \baray{rl}
\min & x_0^6+x_1^6+x_2^6++3 x_1^2 x_2^2x_0^2-x_1^4\left(x_2^2+x_0^2\right) \\
     & \qquad \qquad  -x_2^4 \left(x_1^2+x_0^2\right)-x_0^4\left(x_1^2+x_2^2\right) \\
\st  &x_1^3x_0^2+x_1x_2^4+x_1x_0^4=0,\\
&x_0^2+x_1^2+x_2^2=1,\\
& x_0\geq0,~x_2\geq 0.
\earay \right.
\ee
Its unique minimizer is $(\frac{1}{\sqrt{2}},0,\frac{1}{\sqrt{2}})$
and we can verify that the LICQC, SCC, SOSC hold at it.
By Theorem \ref{finite:poin},
the hierarchy \reff{h3.3}-\reff{hd3.3} has finite convergence.
A numerical experiment by GloptiPoly~3 indicates that $f_3= f^{\prime}_3=-1$.
\end{example}

\subsection{The Moment-SOS hierarchy with denominators}

Let  $K$ be  the feasible set of \reff{po4}.
The Putinar-Vasilescu's Positivstellensatz states that for a polynomial $f$,
if the highest degree homogeneous part $f^\hm$ is a positive definite form
and $f >0$ on $K$, then
\[
(1+\|x\|^2)^kf \in \ideal{c_{eq} } + \qmod{ c_{in} }
\]
for some power $k \in \mathbb{N}$; see \cite{putinar1999positive}.
This motivates the Moment-SOS hierarchy with denominators to solve \reff{po4}:
\be  \label{rel2}
\left\{\baray{rl}
\max  &   \gamma \\
\st  &  \theta^k\left(f- \gamma \right) \in
\ideal{c_{eq} }_{2k} + \qmod{ c_{in} }_{2k} ,
\earay \right.
\ee
where $\theta(x)=1+\|x\|^2$.
Denote by $f_k^\den$ the optimal value of \reff{rel2}.
It was conjectured in \cite{MLM21} that the hierarchy  of \reff{rel2}
has finite convergence under the optimality condition assumptions.
This conjecture was shown to hold in \cite{hny}
when $\ideal{\tilde{c}_{eq}}$ is real radical. We remark that when the degrees of
$f,c_{j}~(j \in \mathcal{I} )$ are all even,
the constraint $x_0\geq0$ is redundant to define the optimization
\reff{h3.5} and the SCC will fail
for the minimizers of \reff{h3.5} with $x_0=0$ (cf.~\cite{hny}).
In this case, we consider the alternative optimization
\be   \label{h3.5local}
\left\{ \baray{rl}
\min  &  F(\tilde{x}) \coloneqq \tilde{f}(\tilde{x}) - f_{min} \cdot x_0^d  \\
\st &   \tilde{c}_{j}(\tilde{x})=0~(j \in \mathcal{E}), \\
&   x_0^2 + \| x \|^2-1=0, \\
&  \tilde{c}_{j}(\tilde{x}) \geq 0~(j \in \mathcal{I}).\\
\earay \right.
\ee
In the following, we show that finite convergence also holds
even if $\ideal{\tilde{c}_{eq}}$ is not real radical.
This resolves the conjecture made in \cite[Section 4.2]{MLM21}
about the finite convergence of \reff{rel2}
under  optimality condition assumptions.

\begin{thm}  	\label{fini:deno}
Let $K$ be the feasible set of \reff{po4}.
\bit

\item [(i)] Suppose $K$ is closed at infinity, the degrees of
$f,c_{j}~(j \in \mathcal{I} )$ are all even, and
$f \in \mR+ int(\mathscr{P}_{d}(K)).$
If the LICQC, SCC and SOSC hold at every minimizer
of \reff{h3.5local}, then $f_k^\den= f_{min}$ for all $k$ big enough.

\item [(ii)]
Suppose $f^\hm$
 is a positive definite form.
If the LICQC, SCC and SOSC hold at every minimizer of \reff{h3.5}, then
$f_k^\den= f_{min}$ for all $k$ big enough.

\eit
\end{thm}
\begin{proof}
(i) Let $d_f \coloneqq d/2$, $d_i \coloneqq \deg(c_j)/2$ $(j\in \mathcal{I})$.
Denote $\tilde{c}_{ie}  \coloneqq  \big\{ \tilde{c_j}(\tilde{x}) \big\}_{j \in \mc{I} } $.
For a degree $k\in \mathbb{N}$, consider the following alternative SOS relaxations
\begin{equation}  \label{hx0}
\left\{ \baray{rl}
\max &   \gamma \\
\st &  \tilde{f}(\tilde{x})- \gamma x_0^{d} \in
\ideal{\tilde{c}_{eq}}_{2k} +\qmod{\tilde{c}_{ie}}_{2k}.
\earay \right.
\end{equation}
 Similar to Theorem \ref{finite:poin}, we know that the hierarchy \reff{hx0} has finite convergence,
under the given assumptions of item (i).
This means that there exists an integer $k_0$ such that
for all $\gamma < f_{min}$, we have  that
\be \label{po:pos:finite}
\tilde{f}-\gamma x_0^{d}\in \ideal{\tilde{c}_{eq}}_{2k_0}+\qmod{\tilde{c}_{ie}}_{2k_0} .
\ee	
Equivalently, there exist $h_0,h_{j}\in \mR[\tilde{x}]_{2k_0}$,
$\sigma_0,~\sigma_{j}\in \Sigma[\tilde{x}]_{2k_0}$ such that
\[
\deg(h_jc_j)\leq 2k_0, \quad \deg(\sigma_jc_j)\leq 2k_0,
\]
\[
\tilde{f}-\gamma x_0^{d}=\sum\limits_{j\in \mathcal{E}}
h_{j}\tilde{c}_j + \sum\limits_{j\in \mathcal{I}}
\sigma_{j}\tilde{c}_j+\sigma_{0}+h_0\cdot (\|\tilde{x}\|^2-1).
\]
The above induces that
\be
\baray{ll}
2(\tilde{f}-\gamma x_0^{d})&=\sum\limits_{j\in \mathcal{E}}(
h_{j}\tilde{c}_j+
h_{j}(-\tilde{x})\tilde{c}_j(-\tilde{x})) + \sum\limits_{j\in \mathcal{I}}
(\sigma_{j}+\sigma_{j}(-\tilde{x}))\tilde{c}_j\\
&+\sigma_{0}+\sigma_{0}(-\tilde{x})+(h_0+h_0(-x))\cdot (\|\tilde{x}\|^2-1).\\
\earay
\ee
if we substitute $\tilde{x}$ for $\frac{(1,x)}{\sqrt{\theta}}$ in the above, then
\be  \nn
\baray{ll}
2\cdot \frac{f(x)-\gamma}{\theta^{d_f}}=& \sum\limits_{j\in \mathcal{E}}
\left[h_{j}\left(\frac{(1,x)}{\sqrt{\theta}}\right)+(-1)^{\deg(c_j)}h_{j}\left(\frac{(-1,-x)}{\sqrt{\theta}}\right)\right]\frac{c_j(x)}{\sqrt{\theta}^{\deg(c_j)}} \\
	&+ \sum\limits_{j\in \mathcal{I}}
	\left[\sigma_{j}\left(\frac{(1,x)}{\sqrt{\theta}}\right)+
\sigma_{j}\left(\frac{(-1,-x)}{\sqrt{\theta}}\right)\right]\frac{c_j(x)}{\theta^{d_j}}\\
	&+\left[\sigma_{0}\left(\frac{(1,x)}{\sqrt{\theta}}\right)+
\sigma_{0}\left(\frac{(-1,-x)}{\sqrt{\theta}}\right)\right].\\
	\earay
\ee
Note that the odd degree terms in $\sigma_{j}\left(\frac{(1,x)}{\sqrt{\theta}}\right)+
\sigma_{j}\left(\frac{(-1,-x)}{\sqrt{\theta}}\right)$ are cancelled.
The above implies that there exists $k_1\in \mathbb{N}$ such that
\[
\theta^{k}(f-\gamma) \in \ideal{c_{eq} }_{2k+2d_f} + \qmod{ c_{in} }_{2k+2d_f}
\]
for all $k\geq k_1$, and $k_1$ is independent of $\gamma$.	
Hence,  $f_k^\den= f_{min}$ for all $k\geq k_1$.

\smallskip \noindent
(ii) For each $j \in \mathcal{I} $, let
$ \theta_j  \coloneqq  2\lceil \frac{\deg(c_j)}{2}\rceil-\deg(c_j).$
Consider the optimization problem
\be \label{opt:PV:eq}
\left\{ \baray{rl}
\min &  \tilde{f}(\tilde{x})  - f_{min} \cdot (x_0)^d \\
\st &   \tilde{c}_{j}(\tilde{x})=0~(j \in \mathcal{E}),\\
& \|\tilde{x}\|^2-1=0,\\
& x_0^{\theta_j} \tilde{c}_{j}(\tilde{x}) \geq 0~(j \in \mathcal{I}).\\
\earay \right.
\ee
Note that the degree of $f$ must be even. Let
$\tilde{u}=(u_0,u)$ be a feasible point of \reff{opt:PV:eq}.
If $u_0=0$, then $\|u\|=1$ and
\[ \tilde{f}(\tilde{u}) - f_{min} \cdot (u_0)^d =f^\hm(u)>0, \]
since $f^\hm$ is positive definite.  If $u_0\neq 0$,
then $u/u_0$ is feasible for \reff{po4} and
\[
\tilde{f}( \tilde{u} )  - f_{min} \cdot (u_0)^d =
u_0^d \big( f(u/u_0) - f_{min} \big) \geq 0 .
\]
Hence, the optimal value of \reff{opt:PV:eq} is zero.
As for Theorem 4.4 of \cite{hny},  the LICQC, SCC and SOSC
hold at every minimizer of \reff{opt:PV:eq}.
The conclusion then follows from item (i).
\end{proof}

\bigskip  \noindent
{\bf Remark:}
Theorem~\ref{finite:poin} (Theorem~\ref{fini:deno}, respectively)
assumes that the LICQC,  SCC and SOSC hold at every minimizer of \reff{h3.5}
(the optimization \reff{h3.5local}, respectively).
The objective function depends on the minimum value $f_{\min}$.
Suppose $\tilde{x}^*=(x_0^*,x^*)$ is a minimizer of \reff{h3.5}.
If $x_0^*>0$, it was shown in \cite{hny} that  $x^*/x_0^*$ is a minimizer of \reff{po4},
and the LICQC, SCC, SOSC hold at  $\tilde{x}^*$ of \reff{h3.5}
if and only if they hold at $x^*/x_0^*$ of \reff{po4}.
For this case, these optimality conditions are independent of the minimum value $f_{\min}$.
If $x_0^*=0$, then the LICQC, SCC and SOSC at $\tilde{x}^*$  for \reff{h3.5}
depend on $f_{\min}$. The same comments hold for the optimization \reff{h3.5local}.
We refer to \cite{hny} for the relationship between their optimality conditions.

\bigskip

The following is an  example to illustrate the finite convergence of \reff{rel2}.

\begin{example}
We consider the polynomial optimization
\be  \label{poex2}
\left\{ \baray{rl}
\min & x_1^2\left(x_1-1\right)^2+(x_2-1)^2\left(x_2-2\right)^2+(x_3+1)^2x_3^2 \\
& \qquad \qquad +2 x_1 (x_2-1) (x_3+1)\left(x_1+x_2+x_3-2\right) \\
\st  &x_1^3-x_2^3-x_3^3-1=0,\\
&(x_1^4+1)(x_1-1)+(x_1^2-x_1)(x_1x_2^2-2x_2)=0.\\
\earay \right.
\ee
The optimal value $f_{min}=0$. One can verify that $x_2+x_3\in \ideal{V_{\mR}(\tilde{c}_{eq})}$
but $x_2+x_3\notin \ideal{\tilde{c}_{eq}}$. Suppose otherwise it were true,
then there exist $h_1,h_2,h_3\in \mR[\tilde{x}]$ such that
\be \nonumber
\baray{ll}
x_2+x_3&=h_1\cdot (x_1^3-x_2^3-x_3^3-x_0^3)+h_2\cdot (x_1^2+x_2^2+x_3^2+x_0^2-1)\\
&+h_3\cdot ((x_1^4+x_0^4)(x_1-x_0)+(x_1^2-x_1x_0)(x_1x_2^2-2x_2x_0^2)),
\earay
\ee
Substituting $(x_0,x_1)$ for $(\frac{1}{\sqrt{2}},\frac{1}{\sqrt{2}})$
in the above, we can get
\[
x_2+x_3 =  -h_1 (\frac{1}{\sqrt{2}},\frac{1}{\sqrt{2}},x_2,x_3)(x_2^3+x_3^3)+h_2
(\frac{1}{\sqrt{2}},\frac{1}{\sqrt{2}},x_2,x_3)  (x_2^2+x_3^2),
\]
which can never hold. So, $\ideal{\tilde{c}_{eq}}$ is not real radical.
The homogenized optimization problem \reff{h3.5} reads
\be  \label{poex2:dual}
\left\{ \baray{rl}
\min & x_1^2\left(x_1-x_0\right)^2+(x_2-x_0)^2\left(x_2-2x_0\right)^2+(x_3+x_0)^2x_3^2 \\
&+2 x_1 (x_2-x_0) (x_3+x_0)\left(x_1+x_2+x_3-2x_0\right) \\
\st  &x_1^3-x_2^3-x_3^3-x_0^3=0,\\
&(x_1^4+x_0^4)(x_1-x_0)+(x_1^2-x_0x_1)(x_1x_2^2-2x_2x_0^2)=0,\\
& x_0^2+x_1^2+x_2^2+x_3^2=1,\\
&x_0\geq0.\\
\earay \right.
\ee
Its unique minimizer is $(\frac{1}{2},\frac{1}{2},\frac{1}{2},-\frac{1}{2})$,
and we can verify that the LICQC, SCC, SOSC hold at it.
By Theorem \ref{fini:deno}, the hierarchy \reff{rel2} has finite convergence.
A numerical experiment by GloptiPoly~3  indicates that $f_3^\den=0$.
\end{example}

\section{The GMP for the super resolution}
\label{sc:su}

GMPs have broad applications. This section gives the application
on the super resolution problem for semialgebraic sets.
The super resolution problem aims to reconstruct high-dimensional sparse vectors from
the observation of a low-pass filter \cite{deca2,deca1,HDLJ}.
This problem can be formulated as the GMP:
\be  \label{su:gpm}
\left\{ \baray{rl}
\min & \langle 1, y\rangle + \langle 1, z\rangle  \\
\st & \langle a_i,y\rangle-\langle a_i,z\rangle   =  b_i \, \, ( 1 \le i \le m),\\
&   ~y, z \in \mathscr{R}_{d}(K) .
\earay \right.
\ee
In the above, $b=(b_1,\dots,b_m)$ is a given vector,
$K$ is the semialgebraic set as in
$\reff{feas:set}$ and $\{a_i\}_{i=1}^m\subseteq \mR[x]_d$
is a set of linearly independent polynomial functions on $K$.
The dual optimization of \reff{su:gpm} is 
\be  \label{su:gpm:dual}
\left\{ \baray{rl}
\max &  b_1\theta_1 + \cdots + b_m \theta_m  \\
\st &  1 - \sum\limits_{i=1}^m \theta_ia_i
\in \mathscr{P}_{d}(K),\\
&  1 + \sum\limits_{i=1}^m \theta_ia_i
\in \mathscr{P}_{d}(K),\\
&  \theta = (\theta_1, \ldots, \theta_m) \in \re^m.
\earay \right.
\ee
For a degree $k \geq  d$, the $k$th moment  relaxation for \reff{su:gpm} is
\be  \label{su:momre:gpm}
\left\{ \baray{cl}
\min &  \langle 1, v\rangle + \langle 1, w\rangle  \\
\st  &  \langle a_i,v\rangle-\langle a_i,w\rangle   =  b_i \, \, ( 1 \le i \le m), \\
& \mathscr{V}_{c_{j}}^{(2k)}[v] = 0,\, \mathscr{V}_{c_{j}}^{(2k)}[w] = 0~ (j\in \mathcal{E}), \\
&L_{c_{j}}^{(k)}[v] \succeq 0,\, L_{c_{j}}^{(k)}[w] \succeq 0~(j\in \mathcal{I}), \\
& M_k[v] \succeq 0, \, M_k[w] \succeq 0, \, v, w \in \mathbb{R}^{\mathbb{N}_{2 k}^{n}} .
\earay \right.
\ee
The dual optimization of \reff{su:momre:gpm} is  the
$k$th order SOS relaxation for \reff{su:gpm:dual}:
\be \label{su:sos:gpm}
\left\{ \baray{cl}
\max &  b_1 \theta_1 + \cdots + b_m \theta_m  \\
\st & 1 - \sum\limits_{i=1}^m \theta_ia_i
\in  \ideal{c_{eq}}_{2k}+\qmod{c_{in}}_{2k},\\
& 1 + \sum\limits_{i=1}^m \theta_ia_i
\in  \ideal{c_{eq}}_{2k}+\qmod{c_{in}}_{2k},\\
& \theta \in \mR^m.
\earay \right.
\end{equation}
Let   $\phi^*$, $\vartheta^*$ denote the optimal value of
\reff{su:gpm}, \reff{su:gpm:dual} respectively,
and let  $\phi_k$, $\vartheta_k$ denote the optimal value of
\reff{su:momre:gpm}, \reff{su:sos:gpm} respectively.
It was observed in \cite[Section 5.2.2]{HDLJ} that the Moment-SOS hierarchy of \reff{su:momre:gpm}--\reff{su:sos:gpm} often has finite convergence.
In the following, we study conditions for its finite convergence.

Suppose $\theta^* = (\theta_1^*, \ldots, \theta_m^*)$
is a maximizer of \reff{su:gpm:dual}.
Consider the optimization problems:
\be  \label{su1:stand:opt}
\left\{ \baray{rl}
\min & 1 - \sum\limits_{i=1}^m \theta_i^* a_i(x)  \\
\st &  	c_{j}(x)=0~(j \in \mathcal{E}), \\
&c_{j}(x) \geq 0~(j \in \mathcal{I}),
\earay \right.
\ee
and
\be  \label{su2:stand:opt}
\left\{ \baray{rl}
\min & 1+ \sum\limits_{i=1}^m \theta_i^* a_i(x)  \\
\st &  	c_{j}(x)=0~(j \in \mathcal{E}), \\
&c_{j}(x) \geq 0~(j \in \mathcal{I}).
\earay \right.
\ee
Note that the objective polynomials in 
\reff{su1:stand:opt} and \reff{su2:stand:opt}
are both nonnegative on the set $K$.

\begin{thm} \label{su:finite:gpm}
Suppose $\ideal{c_{eq}}+\qmod{c_{in}}$ is archimedean and
$\theta^* = (\theta_1^*, \ldots, \theta_m^*)$
is a maximizer of \reff{su:gpm:dual}.
If the LICQC, SCC and SOSC hold at every minimizer of
\reff{su1:stand:opt} and \reff{su2:stand:opt},
then we have:

\bit

\item [(i)]
The hierarchy of \reff{su:momre:gpm}--\reff{su:sos:gpm} has finite convergence, i.e.,
$\vartheta_k = \phi_k = \vartheta^* = \phi^*$ for all $k$ big enough.	

\item [(ii)]
Every minimizer of the moment relaxation \reff{su:momre:gpm}
must have a flat truncation, when $k$ is sufficiently large.

\eit
\end{thm}

\begin{proof}
(i) Note that for $\theta=(0,\dots,0)$, we have
\[
1-\sum\limits_{i=1}^m \theta_i a_i(x)=1+\sum\limits_{i=1}^m
\theta_i a_i(x)=1\in int(\mathscr{P}_d(K)).
\]
Thus, the Slater condition holds for the pair \reff{su:gpm}--\reff{su:gpm:dual}.
By following the same proof of Theorem \ref{finite:gpm} (i),
we can show that for all $\eps >0$,
there exists $k_0\in \mathbb{N}$, independent of $\epsilon$, such that 
\begin{eqnarray*}
1-  \sum\limits_{i=1}^m \theta_i^* a_i(x) + \eps & \in &
\ideal{c_{eq}}_{2k_0}+\qmod{c_{in}}_{2k_0},  \\
1+ \sum\limits_{i=1}^m \theta_i^* a_i(x) + \eps & \in &
\ideal{c_{eq}}_{2k_0}+\qmod{c_{in}}_{2k_0}.
\end{eqnarray*}
 Thus, we know
$\frac{1}{1+\epsilon} \cdot \theta^*$ is feasible for
\reff{su:sos:gpm} at the relaxation order $k_0$,
for all $\epsilon >0$. As $\epsilon \rightarrow 0$, we get
\be
b^{\mathrm{T}} ( \frac{1}{1+\epsilon}\cdot \theta^* )
= \frac{1}{1+\epsilon}  b^{\mathrm{T}}\theta^*
\rightarrow   b^{\mathrm{T}} \theta^* .
\ee
This implies that $\vartheta_{k_0} =  b^{\mathrm{T}} \theta^* = \vartheta^*$ and then
$\vartheta^*=\phi^*=\vartheta_{k_0} = \phi_{k_0} $.
Hence, the Moment-SOS hierarchy  of
\reff{su:momre:gpm}--\reff{su:sos:gpm} has finite convergence.

(ii) Suppose $(v^{(k)}, w^{(k)})$ is a minimizer of
\reff{su:momre:gpm} for the relaxation order $k$.
The $k$th order moment relaxations for
polynomial optimization problems
\reff{su1:stand:opt} and \reff{su2:stand:opt} are respectively
\be  \label{su1mom:opt:stan}
\left\{ \baray{cl}
\min &  \langle 1 - \sum\limits_{i=1}^m \theta_i^* a_i  , v \rangle  \\
\st  &  \langle 1,v\rangle  =  1,  \\
&\mathscr{V}_{c_{j}}^{(2k)}[v] =  0~ (j\in \mathcal{E}), \\
&L_{c_{j}}^{(k)}[v] \succeq 0~(j\in \mathcal{I}), \\
& M_k[v] \succeq 0, \, v \in \mathbb{R}^{\mathbb{N}_{2 k}^{n}},
\earay \right.
\ee	
and
\be  \label{su2mom:opt:stan}
\left\{ \baray{cl}
\min &  \langle 1 + \sum\limits_{i=1}^m \theta_i^* a_i  , w \rangle  \\
\st  &  \langle 1,w\rangle  =  1,  \\
&\mathscr{V}_{c_{j}}^{(2k)}[w] =  0~ (j\in \mathcal{E}), \\
&L_{c_{j}}^{(k)}[w] \succeq 0~(j\in \mathcal{I}), \\
& M_k[w] \succeq 0, \, w \in \mathbb{R}^{\mathbb{N}_{2 k}^{n}} .
\earay \right.
\ee	
It follows from (i) that for all $k\geq k_0$,
\[
\langle 1, v^{(k)}\rangle + \langle 1, w^{(k)}\rangle=
\sum\limits_{i=1}^m b_i \theta_i^* =
\sum\limits_{i=1}^m (\langle a_i,v^{(k)}\rangle-\langle a_i,w^{(k)}\rangle)\theta_i^*.
\]
Hence, we get
\be \label{sueq1}
\langle 1-\sum\limits_{i=1}^m \theta_i^*a_i, v^{(k)}\rangle +
\langle 1-\sum\limits_{i=1}^m \theta_i^*a_i, w^{(k)}\rangle=0.
\ee
In the item (i), we have shown that
\[
\Big\{ 1 - \sum\limits_{i=1}^m \theta_i^* a_i,\, 1 + \sum\limits_{i=1}^m \theta_i^* a_i
\Big\} \subseteq   cl(\ideal{c_{eq}}_{2k_0}+\qmod{c_{in}}_{2k_0}) .
\] 
Since $v^{(k)}, w^{(k)}$ are feasible for \reff{su1mom:opt:stan}, \reff{su2mom:opt:stan} respectively, we have
\[
\langle 1-\sum\limits_{i=1}^m \theta_i^*a_i, v^{(k)}\rangle\geq 0,\quad
\langle 1-\sum\limits_{i=1}^m \theta_i^*a_i, w^{(k)}\rangle \geq 0.
\]
Combining with \reff{sueq1},  we get
\[
\langle 1 - \sum\limits_{i=1}^m \theta_i^* a_i, v^{(k)}\rangle=0,\quad
\langle 1 + \sum\limits_{i=1}^m \theta_i^* a_i, w^{(k)}\rangle=0.
\]
We show that $v^{(k)}, w^{(k)}$ both have flat truncations,
in four cases.

\bigskip
\noindent
{\it Case I:} $\big( v^{(k)} \big)_0= \big( w^{(k)} \big)_0=0$.
Then, for all $|\alpha| \leq 2 k-2$, we have $(v^{( k)})_{\alpha}=0$,
$(w^{( k)})_{\alpha}=0$.
This can be implied by Lemma~5.7 of \cite{Lau09}.
Thus, the truncations $\left.v^{( k)}\right|_{2k-2}$
and $\left.w^{( k)}\right|_{2 k-2}$ are flat.

\medskip 
\noindent
{\it Case II:} $\big( v^{(k)} \big)_0=0$, $\big( w^{(k)} \big)_0 > 0$.
As for case I, the truncation $\left.v^{( k)}\right|_{2 k-2}$ is flat.
Since $\big( w^{(k)} \big)_0 > 0$ and $\langle 1 + \sum\limits_{i=1}^m \theta_i^* a_i, w^{(k)}\rangle=0$, the normalization $w^{( k)} /w^{(k)}_0$ 
is a minimizer of \reff{su2mom:opt:stan}  for all $k\geq k_0$.
It follows from Theorem \ref{poly:fin} that  $w^{(k)}$ 
has a flat truncation, when $k$ is sufficiently large.

\medskip 
\noindent
{\it Case III:} $\big( v^{(k)} \big)_0 > 0$, $\big( w^{(k)} \big)_0 = 0$. The proof is the same as for Case II.

\medskip 
\noindent
{\it Case IV:} $\big( v^{(k)} \big)_0 > 0$, $\big( w^{(k)} \big)_0 > 0$.
Then,  $v^{(k)}/ \big( v^{(k)} \big)_0$ is a minimizer of \reff{su1mom:opt:stan} 
and $w^{(k)}/ \big( w^{(k)} \big)_0$
is a minimizer of \reff{su2mom:opt:stan}  for all $k\geq k_0$.
By Theorem \ref{poly:fin}, we know that both $v^{(k)}$ and $w^{(k)}$ have flat truncations for  sufficiently large $k$.

\end{proof}

\section{Conclusions and discussions}
\label{sc:con}

In this paper, we prove the finite convergence of the  Moment-SOS hierarchy
for solving generalized moment problems
under the archimedeanness and optimality conditions,
but without real radicalness of the equality constraint ideal.
This improves the finite convergence theory in the earlier work.
When the constraint set $K$ is unbounded (the archimedeanness fails in this case),
we propose a homogenized Moment-SOS hierarchy
and prove similar finite convergence results.
The applications of these results in
polynomial optimization with unbounded sets
are also discussed.

There still exist many interesting problems for future work.
One of them is about finite convergence of the Moment-SOS hierarchy
for generic GMPs. To be more specific, will the finite convergence always hold
except a zero-measure set in the space of input polynomials?
We have shown the finite convergence under the archimedeanness
and optimality conditions for \reff{stand:opt}.
%
%
To the best of  authors' knowledge,
it is an open question whether or not the LICQC, SCC, SOSC hold
at every minimizer of \reff{stand:opt},
when the polynomials $f, a_i, b_i, c_j$ have generic coefficients.
Note that $\theta^*=(\theta_1^*,\dots,\theta_m^*)$ is an optimizer of \reff{gpm:dual}
and it is not generic but determined by $f, a_i, b_i, c_j$.
Therefore, we pose the following conjecture.

\begin{conj} \label{GMP:conj}
Let $d_0$ and $d_i$ be positive degrees, for $i \in [m] \cup \mc{E} \cup \mc{I}$.
Consider $f\in \mathbb{R}[x]_{d_0}$, $a_i\in \mathbb{R}[x]_{d_i}$ $(i\in [m])$,
$b \in \re^m$, $c_j\in \mathbb{R}[x]_{d_j}$ $(j \in \mc{E} \cup \mc{I})$.
Then  there exists a finite set of polynomials $\varphi_1, \ldots, \varphi_L$,
which are in the coefficients of polynomials $f$,
$a = (a_i)_{i \in [m]}$, $b$,  $c = (c_j)_{ j \in \mc{E} \cup \mc{I} }$
 such that if
\[
\varphi_k\left(f, a ,b, c \right)\neq 0, ~k=1\dots,L,
\]
then the Moment-SOS hierarchy \reff{momre:gpm}--\reff{sos:gpm} has finite convergence.
\end{conj}

\bigskip 
\noindent
{\bf Acknowledgements.}
The authors would like to thank the editors and anonymous referees
for their fruitful suggestions on the paper.


\begin{thebibliography}{99}

%
%


%
%



\bibitem{bmmp}
 L. Baldi and  B. Mourrain,
\newblock{\em  On the effective Putinar's Positivstellensatz and moment
approximation},
\newblock Math. Program.,
200 (2023), pp. 71–103.



\bibitem{Bar02}
A. Barvinok,
\newblock{\em A Course in Convexity},
\newblock American Mathematical Society,
Providence, 2002.


\bibitem{Bert97}
  D. Bertsekas,
\newblock{\em Nonlinear Programming, 2nd edn},
\newblock  Athena Scientific,
 Belmont, 1995.

\bibitem{bhl}
 D. Brake,  J. Hauenstein, and  A. Liddell,
\newblock{\em Validating
the completeness of the real solution set of a system of polynomial equations},
\newblock {in Proceedings of the the ISSAC ’16, ACM, New
	York,  2016,  pp. 143--150. }









\bibitem{CF05}
 R. Curto and   L. Fialkow,
\newblock{\em Truncated K--moment problems in several variables},
\newblock J. Oper. Theory 54 (2005), pp. 189--226.




\bibitem{deca2}
Y. De Castro and F. Gamboa,
\newblock {\em  Exact reconstruction using Beurling minimal extrapolation},
\newblock {J. Math. Anal. Appl.},
395 (2012), pp. 336--354.


\bibitem{deca1}
Y. De Castro, F. Gamboa, D. Henrion, and J. Lasserre,
\newblock {\em  Exact solutions to super resolution on semi-algebraic domains in higher dimensions},
\newblock {IEEE Trans. Inform. Theory},
 63 (2017), pp. 621--630.







\bibitem{dKlLau11}
 E. de Klerk and   M. Laurent,
\newblock {\em On the Lasserre hierarchy of semidefinite programming relaxations
of convex polynomial optimization problems},
\newblock {SIAM J. Optim.},
21 (2011), pp. 824--832.



\bibitem{klerk2019survey}
 E. de Klerk and   M. Laurent,
\newblock{\em A survey of semidefinite programming approaches
to the generalized problem of moments and their error analysis},
\newblock in World Women in Mathematics 2018:
Proceedings of the First World Meeting for Women in Mathematics (WM),
C. Araujo, G. Benkart, C. E. Praeger, and B. Tanbay, eds.,
Springer, Cham, 2019, pp. 17--56.






\bibitem{DNY22}
 M. Dressler,  J. Nie, and  Z. Yang,
\newblock {\em Separability of Hermitian tensors and PSD decompositions},
\newblock  Linear  Multilinear Algebra, 70 (2022), pp. 6581--6608.



\bibitem{FF20}
 K. Fang and  H. Fawzi,
\newblock{\em The sum-of-squares hierarchy on the sphere and applications in quantum information theory},
\newblock Math. Program., 190 (2021), pp. 331--360.







%
%




\bibitem{gls}
 S. Gribling,  M. Laurent, and  A. Steenkamp,
\newblock{\em Bounding the separable rank via polynomial optimization},
\newblock {Linear Algebra Appl.}, 648 (2022), pp. 1--55.


\bibitem{guo2014minimizing}
 F. Guo,  L. Wang,  G. Zhou:
\newblock Minimizing rational functions by exact Jacobian SDP relaxation
applicable to finite singularities.
\newblock {J. Global Optim.},
 58 (2014), pp. 261--284.



\bibitem{2009GloptiPoly}
D. Henrion, J. Lasserre, and   J. Loefberg,
\newblock{\em Gloptipoly 3: moments, optimization and semidefinite programming},
\newblock {Optim. Methods Softw.}, 24 (2009), pp. 761--779.


\bibitem{HDLJ}
D. Henrion,  M. Korda, and  J. Lasserre,
\newblock{\em The Moment-SOS Hierarchy: Lectures In Probability, Statistics,
Computational Geometry, Control And Nonlinear Pdes},
\newblock  World Scientific, Cambridge, 2020.




\bibitem{HilNie08}
 C. Hillar and J.  Nie,
{\em An elementary and constructive proof of Hilbert's 17th Problem for matrices},
{Proc. Amer. Math. Soc.}, 136 (2008), pp. 73--76.




\bibitem{hny}
L. Huang,  J. Nie, and  Y. Yuan,
\newblock{\em Homogenization for polynomial optimization with unbounded sets},
\newblock {Math. Program.}, 200 (2023), pp. 105--145.


\bibitem{hny1}
L. Huang,  J. Nie, and  Y. Yuan,
\newblock{\em Generalized truncated moment problems with unbounded sets},
\newblock {J. Sci. Comput.,} 95
(2023).



\bibitem{kfe}
 F. Kirschner and  E. de Klerk,
\newblock{\em  Convergence rates of RLT and Lasserre-type hierarchies
for the generalized moment problem over the simplex and the sphere},
\newblock {Optim. Lett.}, 16  (2022), pp. 2191--2208.




\bibitem{klm}
 M. Korda,  M. Laurent,  V. Magron, and A.  Steenkamp,
\newblock{\em Exploiting ideal-sparsity in the generalized moment problem with application to matrix factorization ranks},
\newblock {Math. Program.},  (2023).


%
%


\bibitem{Las01}
 J. Lasserre,
\newblock{\em Global optimization with polynomials and the problem of moments},
\newblock{SIAM J. Optim.}, 11 (2001), pp. 796--817.


\bibitem{LLR08}
J. Lasserre,  M. Laurent, and   P. Rostalski,
\newblock{\em Semidefinite characterization and computation
of zero-dimensional real radical ideals},
\newblock { Found. Comput.	Math.},
8 (2008), pp. 607--647.


\bibitem{Las08}
J. Lasserre,
\newblock{\em A semidefinite programming approach to the generalized problem of moments},
\newblock { Math. Program.},
112 (2008), pp. 65--92.




\bibitem{Las09}
J. Lasserre,
\newblock{\em Convexity in semi-algebraic geometry and polynomial optimization},
\newblock {SIAM J. Optim.}, 19 (2009), pp. 1995--2014.






\bibitem{LasBk15}
 J. Lasserre,
\newblock{\em An Introduction to Polynomial and Semi-algebraic Optimization},
\newblock  Cambridge
University Press, Cambridge, 2015.




\bibitem{Lau05}
M. Laurent,
\newblock{\em Revisiting two theorems of Curto and Fialkow on moment matrices},
\newblock{Proc. Am. Math. Soc.}, 133 (2005), pp. 2965--2976.




\bibitem{Lau07}
M. Laurent,
\newblock{\em Semidefinite representations for finite varieties},
\newblock {Math. Program.},  109 (2007), pp. 1--26.


\bibitem{Lau09}
 M. Laurent,
\newblock {\em Sums of squares, moment matrices and optimization over polynomials},
\newblock in { Emerging applications of algebraic geometry of IMA Volumes in Mathematics and its Applications},   IMA Vol. Math. Appl. 149, Springer, New York, 2009, pp. 157--270.







\bibitem{MLM21}
 N. Mai,  J. Lasserre, and  V. Magron,
\newblock{\em Positivity certificates and polynomial optimization on non-compact
  semialgebraic sets},
\newblock {{Math. Program.}},  194 (2022), pp. 443–485.


\bibitem{maim}
 N. Mai and   V. Magron,
\newblock{\em On the complexity of Putinar–Vasilescu's Positivstellensatz},
\newblock {J. Complexity}, 72  (2022).



%
%



\bibitem{mar06}
 M. Marshall,
\newblock  {\em Representation of non-negative polynomials with finitely many zeros},
\newblock {Ann. Faculte Sci. Toulouse},
15 (2006), pp. 599--609.


%
%


\bibitem{marshall2008positive}
 M.  Marshall,
\newblock{\em Positive Polynomials and Sums of Squares},
\newblock American Mathematical Society,
Providence, 2008.



\bibitem{mona}
 O. Mula, and  A. Nouy,
\newblock{\em Moment-SoS Methods for Optimal Transport Problems},
\newblock  \emph{preprint},
 \url{arXiv:2211.10742}, 2022.





\bibitem{NiePMI}
 J. Nie,
{\em Polynomial matrix inequality and semidefinite representation},
{Math. Oper. Res.}, 36 (2011), pp. 398--415.


\bibitem{njwdis}
J. Nie,
\newblock{\em Discriminants and nonnegative polynomials},
\newblock {J. Symb. Comput.}, 47 (2012), pp. 167--191.



\bibitem{NieSOSbd}
 J. Nie,
{\em Sum of squares methods for minimizing polynomial forms over spheres and hypersurfaces},
{Front. Math. China},
7 (2012), pp. 321--346.



\bibitem{nie2013certifying}
 J. Nie,
\newblock{\em Certifying convergence of Lasserre’s hierarchy via flat truncation},
\newblock{Math. Program.}, 142 (2013), pp. 485--510.



\bibitem{Nie13}
 J. Nie,
\newblock{\em Polynomial optimization with real varieties},
\newblock {SIAM J. Optim.}, 23 (2013), pp. 1634--1646.



\bibitem{nieopcd}
 J. Nie,
\newblock{\em Optimality conditions and finite convergence of Lasserre’s
  hierarchy},
\newblock {Math. Program.}, 146 (2014), pp. 97--121.



\bibitem{STNN17}
 J. Nie,
{\em Symmetric tensor nuclear norms},
{SIAM J. Appl. Algebra Geom.},
1 (2017), pp. 599--625.


\bibitem{niebook}
J. Nie,
\newblock{\em Moment and Polynomial Optimization},
\newblock  SIAM, Philadelphia, 2023.




\bibitem{niesch}
J. Nie and  M. Schweighofer,
\newblock{\em On the complexity of Putinar's Positivstellensatz},
\newblock {J. Complexity} 23 (2007), pp. 135--150.


\bibitem{nieyz}
J. Nie,   L. Yang,  S. Zhong, and  G. Zhou,
\newblock{\em Distributionally robust optimization with moment
	ambiguity sets},
\newblock {J. Sci. Comput.} 94 (2022).


\bibitem{NZ16}
J. Nie and   X. Zhang,
\newblock{\em Positive maps and separable matrices},
\newblock SIAM J. Optim. 26 (2016), pp. 1236--1256.


\bibitem{NieZhang18}
 J. Nie and  X. Zhang,
{\em Real eigenvalues of nonsymmetric tensors},
{Comp. Opt.  Appl.},
70 (2018), pp. 1--32.





\bibitem{NieZhong23}
 J. Nie and  S. Zhong,
\newblock{\em Distributionally robust optimization with polynomial robust constraints},
\newblock preprint,
\url{arXiv:2308.15591},  2023.


\bibitem{putinar1993positive}
  M. Putinar,
\newblock{\em Positive polynomials on compact semi-algebraic sets},
\newblock {Indiana Univ. Math. J.}, 42 (1993), pp. 969--984.

\bibitem{putinar1999positive}
 M. Putinar and   F. Vasilescu,
\newblock{\em Positive polynomials on semi-algebraic sets},
\newblock {C. R. Acad. Sci. Ser. I Math.}, 328 (1999), pp. 585--589.

%
%





\bibitem{Sch09}
 C. Scheiderer,
\newblock{\em Positivity and sums of squares: A guide to recent results},
  in { Emerging applications of algebraic geometry of IMA Volumes in Mathematics and its Applications},   IMA Vol. Math. Appl. 149, Springer, New York, 2009, pp. 271--324.


%
%



\bibitem{slot1}
 L. Slot and  M. Laurent,
\newblock{\em Sum-of-squares hierarchies for binary polynomial optimization},
\newblock {Math. Program.,}  197 (2023), pp. 621-660.


%
%

\end{thebibliography}
\end{document}